\documentclass[12pt, a4paper]{amsart}

\usepackage[hmargin=30mm, vmargin=25mm, includefoot, twoside]{geometry}
\usepackage[bookmarksopen=true]{hyperref}

\usepackage{amscd}
\usepackage[all,pdf]{xy}

\usepackage{amsfonts,amssymb,verbatim}
\usepackage{latexsym}
\usepackage{mathrsfs}
\usepackage{stmaryrd}
\usepackage{xspace}
\usepackage{enumerate, paralist}
\usepackage{graphicx}
\usepackage[all]{xy}
\usepackage{extarrows}
\usepackage[usenames,dvipsnames]{color}
\usepackage{txfonts, pxfonts}
\usepackage{amsthm}
\usepackage{amsmath}
\usepackage{tikz-cd}

\usepackage{todonotes}
\newtheorem{thm}{Theorem}[section]
\newtheorem{cor}[thm]{Corollary}
\newtheorem{lem}[thm]{Lemma}
\newtheorem{prop}[thm]{Proposition}

\newtheorem{introthm}{Theorem}

\theoremstyle{definition}
 \newtheorem{introdefn}[introthm]{Definition}

\numberwithin{equation}{section}

\theoremstyle{definition}
\newtheorem{defn}[thm]{Definition}
\theoremstyle{remark}
\newtheorem{rem}[thm]{Remark}

\def\N{\mathbb{N}}
\def\Nd{\mathcal{N}}
\def\C{\mathbb{C}}
\def\CC{\mathcal{C}}
\def\R{\mathbb{R}}
\def\B{\mathfrak{B}}
\def\H{\mathcal{H}}
\def\K{\mathfrak{K}}
\def\HH{\mathscr{H}}
\def\supp{\mathrm{supp}}
\def\pro{\mathrm{prop}}
\def\Id{\mathrm{Id}}
\def\Cliff{\mathrm{Cliff}}
\def\diam{\mathrm{diam}}
\def\Ad{\mathrm{Ad}}
\def\cs{C^*_{sq}}
\def\c{C^*}
\def\cq{C^*_q}
\def\A{\mathcal{A}}
\def\sL{C^*_{L,sq}}
\def\L{C^*_L}
\def\qL{C^*_{L,q}}

\def\Q{\mathcal{Q}((\mathcal{A}_n\otimes\mathfrak{K})_{n})}
\def\q{\mathcal{Q}((\mathcal{A}_n)_{n})\otimes\mathfrak{K}}

\def\2Q{\mathcal{Q}((M_2(\mathcal{A}_n))_{n\in\mathbb{N}})\otimes\mathfrak{K}}
\def\+Q{\mathcal{Q}((M_2(\mathcal{A}_n^+))_{n\in\mathbb{N}})\otimes\mathfrak{K}}

\tikzcdset{scale cd/.style={every label/.append style={scale=#1},
		cells={nodes={scale=#1}}}}

\begin{document}

\title[The strongly quasi-local coarse Novikov conjecture]{The strongly quasi-local coarse Novikov conjecture and Banach spaces with Property (H)}

\author{Xiaoman Chen}
\author{Kun Gao}
\author{Jiawen Zhang}

\address{School of Mathematical Sciences, Fudan University, 220 Handan Road, Shanghai, 200433, China.}
\email{xchen@fudan.edu.cn, 19110180024@fudan.edu.cn, jiawenzhang@fudan.edu.cn}

\thanks{KG was supported by the National Natural Science Foundation of China under Grant (No.12071183).}

\date{}


\baselineskip=16pt

\begin{abstract}
In this paper, we introduce a strongly quasi-local version of the coarse Novikov conjecture, which states that certain assembly map from the coarse $K$-homology of a metric space to the $K$-theory of its strongly quasi-local algebra is injective. We prove that the conjecture holds for metric spaces with bounded geometry which can be coarsely embedded into Banach spaces with Property (H), introduced by Kasparov and Yu. Besides, we also generalise the notion of strong quasi-locality to proper metric spaces and provide a (strongly) quasi-local picture for $K$-homology.
\end{abstract}

\date{\today}
\maketitle

\parskip 4pt


\section{Introduction}

The coarse Novikov conjecture \cite{HR93, KY06, Yu95, yu1998:Novikov-for-FAD} asserts that the coarse Baum-Connes assembly map from the coarse $K$-homology $KX_\ast(X)$ of a metric space $X$ to the $K$-theory of its Roe algebra $C^*(X)$ is injective. Here the Roe algebra $C^*(X)$, encoding the coarse geometry of the space $X$, was introduced by John Roe \cite{roe1988:index-thm-on-open-mfds, Roe96} in his pioneering work to study index theory on open manifolds. 

The coarse Novikov conjecture is a geometric analogue of the strong Novikov conjecture \cite{BCH94} and suggests an approach to detect the higher index of the Dirac operator on a noncompact complete Riemannian manifold, which has many significant applications in geometry and topology. In particular, it implies the Gromov conjecture on non-existence of positive scalar curvature metrics on uniformly contractible Riemannian manifolds, and the zero-in-the-spectrum conjecture concerning the spectrum of the Laplacian operator.

Over the last two decades, there are a lot of remarkable progresses on the coarse Novikov conjecture. For example, Yu proved the coarse Baum-Connes conjecture, and consequently the coarse Novikov conjecture, for metric spaces with bounded geometry which admit a coarse embedding into a Hilbert space \cite{Yu00}. Later on, Kasparov and Yu proved the coarse Novikov conjecture for metric spaces with bounded geometry which can be coarsely embedded into a uniformly convex Banach space \cite{KY06}. In \cite{GWY08}, Gong, Wang and Yu introduced the maximal Roe algebras and studied a maximal version of the coarse Novikov conjecture. So far, no counterexample to the coarse Novikov conjecture has been discovered.

Property (H) is a geometric condition to Banach spaces, introduced by Kasparov and Yu in \cite{KY12} where they proved the strong Novikov conjecture for groups which can be coarsely embedded into a Banach space with Property (H). Later on, Wang, Yu and the first author proved the coarse Novikov conjecture for metric spaces with bounded geometry which can be coarsely embedded into a Banach space with Property (H) \cite{CWY15}.


\begin{introdefn}[\cite{KY12}]\label{defn:CWY15}
A real Banach space $ V $ is said to have Property (H) if there exist increasing sequences of finite dimensional subspaces $ \{V_n\}_{n\in\N} $ of $V$ and $ \{W_n\}_{n\in\N} $ of a real Hilbert space $W$ such that:
\begin{enumerate}
 \item $\bigcup_{n=1}^\infty V_n$ is dense in $V$;
 \item there exists a uniformly continuous map $ \psi:S(\bigcup_{n=1}^\infty V_n) \rightarrow S(\bigcup_{n=1}^\infty W_n) $ such that $\psi|_{S(V_n)}$ is a homeomorphism onto $ S(W_n) $ for any $n\in\N$, where $S(\cdot)$ denotes the unit sphere of a subspace of a Banach space.
\end{enumerate}
\end{introdefn}

For example, the Banach space $\ell^p(\N)$ and the Banach space of all Schattern-$p$ class operators on a separable Hilbert space have Property (H) for $p \geq 1$ due to the Mazur(-type) map. One can also check that uniformly convex Banach spaces with unconditional basis have Property (H) (see, \emph{e.g.}, \cite[Chapter 9]{BL00}).

In this paper, we introduce a strongly quasi-local version of the coarse Novikov conjecture and prove it for metric spaces with bounded geometry which can be coarsely embedded into a Banach space with Property (H). 

Recall that the notion of quasi-locality was introduced by Roe as an intrinsic characterisation for elements in the Roe algebra \cite{roe1988:index-thm-on-open-mfds}. It turns out that the quasi-local algebra $C^*_q(X)$ of a metric space $X$ (\emph{i.e.}, the $C^*$-algebra of all locally compact and quasi-local operators on $X$) contains the Roe algebra $C^*(X)$ as a $C^*$-subalgebra, and they coincide when the space $X$ has Yu's property A \cite{li2021quasi, SZ20}. On the other hand, Engel \cite{engel2015rough} discovered that while indices of genuine differential operators on Riemannian manifolds live in the $K$-theory of (appropriate) Roe algebras, the indices of uniform pseudo-differential operators are only known to be in the $K$-theory of quasi-local algebras. Hence it is important to study whether the Roe algebra and the quasi-local algebra have the same $K$-theory.

As a potential approach, the first and third authors together with Bao introduced the following notion of strong quasi-locality in \cite{quasi-local} (see Section \ref{sec:sq algebras} for more details). Fix an infinite-dimensional separable Hilbert space $\HH$ and denote $\K(\HH)_1$ the unit ball of the compact operators on $\HH$.

\begin{introdefn}[\cite{quasi-local}]\label{introdefn:strong quasi-locality}
Let $X$ be a discrete metric space with bounded geometry and $T \in \B(\ell^2(X; \HH))$. We say that $T$ is \emph{strongly quasi-local} if for any $\varepsilon > 0$ there exist $\delta, R >0$ such that for any map $g:X \rightarrow \K(\HH)_1$ satisfying that $d(x,y)< R $ implies $\|g(x) - g(y)\|<\delta$, we have
\begin{equation*}
\big\| [T\otimes \Id_{\HH} , \Lambda(g) ] \big\|< \varepsilon
\end{equation*}
where $\Lambda(g) \in \B(\ell^2(X; \HH \otimes \HH))$ is defined by $\Lambda(g)(\delta_x \otimes \xi \otimes \eta):=\delta_x \otimes \xi \otimes g(x)\eta$ for $\delta_x \otimes \xi \otimes \eta \in \ell^2(X; \HH \otimes \HH) \cong \ell^2(X) \otimes \HH \otimes \HH$.
\end{introdefn}

The strongly quasi-local algebra $C^*_{sq}(X)$ for a discrete metric space $X$ is defined to be the $C^*$-algebra of all locally compact and strongly quasi-local operators on $X$. The strongly quasi-local algebra $C^*_{sq}(X)$ contains the Roe algebra $C^*(X)$, and the main result of \cite{quasi-local} states that they have the same $K$-theory when $X$ has bounded geometry and can be coarsely embedded into a Hilbert space.

The main focus of this paper is the following strongly quasi-local version of the coarse Novikov conjecture:

\noindent\textbf{The strongly quasi-local coarse Novikov conjecture.} If $X$ is a discrete metric space with bounded geometry, then the following \emph{strongly quasi-local coarse assembly map} $\mu_{sq}:=i_\ast \circ \mu$ is injective:
\[
\mu_{sq}: KX_\ast(X) \stackrel{\mu}{\longrightarrow} K_\ast(C^*(X)) \stackrel{i_\ast}{\longrightarrow} K_\ast(C^*_{sq}(X))
\]
where $\mu$ is the coarse Baum-Connes assembly map.

One can also consider the strongly quasi-local coarse Baum-Connes conjecture which asserts that the above map $\mu_{sq}$ is bijective. Combining with Yu's result \cite{Yu00}, the main result of \cite{quasi-local} can be restated that the strongly quasi-local coarse Baum-Connes conjecture holds for metric spaces with bounded geometry which can be coarsely embedded into a Hilbert space.

The following is the main result of this paper:

\begin{introthm}\label{thm:main result}
	Let $X$ be a bounded geometry metric space which can be coarsely embedded into a Banach space with Property (H). Then the strongly quasi-local coarse Novikov conjecture holds for $X$.
\end{introthm}

The proof of Theorem \ref{thm:main result} is inspired by that of \cite[Theorem 1.1]{CWY15}, but is more involved and requires new techniques. Amongst other pieces, the key ingredient in the proof is that we introduce a twisted quasi-local algebra and construct a Bott map from the $K$-theory of strongly quasi-local algebra to that of the twisted quasi-local algebra. Recall that in the case of Roe algebras, a Bott map was constructed in \cite{CWY15} thanks to the finite propagation approximations for operators in Roe algebras. However for general operators in the strongly quasi-local algebra, it is unclear whether one can still find finite propagation approximations. To overcome this issue, we make use of the hypothesis of strong quasi-locality together with some technical arguments to obtain the required Bott map.

On the other hand, we also provide a (strongly) quasi-local picture for $K$-homology. Recall that Yu \cite{Yu97} introduced a notion of localisation algebra $C^*_L(X)$ for a proper metric space $X$ and showed that its $K$-theory is isomorphic to the $K$-homology of $X$ (see also \cite{QR10}). Note that operators in $C^*_L(X)$ have smaller and smaller propagation as the parameter tends to infinity. Following the same idea, we introduce a notion of localisation (strongly) quasi-local algebra where the ``extent'' of (strong) quasi-locality of operators therein get better and better as the parameter tends to infinity (see Definition \ref{defn:localisation quasi-local Roe algebra} and \ref{defn:localisation strongly quasi-local Roe algebra}). We prove the following:

\begin{introthm}\label{introthm:localisation}
Let $X$ be a proper metric space. The localisation algebra, the localisation quasi-local algebra and the localisation strongly quasi-local algebra have the same $K$-theory, which coincides with the $K$-homology of $X$.
\end{introthm}

Consequently, we obtain a (strongly) quasi-local description of the coarse $K$-homology. Moreover, the strongly quasi-local coarse assembly map $\mu_{sq}$ above is induced by an evaluation map (see Corollary \ref{cor:equiv for sqcbc}). This is also based on our generalised notion of strong quasi-locality for proper metric spaces (note that the original one introduced in \cite{quasi-local} is only for discrete metric spaces).

The paper is organised as follows. In Section \ref{sec:pre}, we collect notions from coarse geometry and recall the definitions of Roe and quasi-local algebras. In Section \ref{sec:sq algebras}, we generalise the notion of strong quasi-locality to the case of proper metric spaces and verify their coarse geometric properties, based on which we introduce a (strong) quasi-local picture for $K$-homology and prove Theorem \ref{introthm:localisation} in Section \ref{sec:localisation}. The rest of the paper is contributed to the proof of Theorem \ref{thm:main result}. More precisely, in Section \ref{sec:twisted alg} we recall the twisted Roe algebras and introduce their quasi-local counterparts for metric spaces which can be coarsely embedded into a Banach space with Property (H). Then in Section \ref{sec:Bott} we construct Bott maps to link the $K$-theories of Roe and strongly quasi-local algebras with those of their twisted counterparts, and prove that these twisted algebras have the same $K$-theory (Theorem \ref{thm:isomorphism between twisted Roe algebra and twisted quasi-local algebra in K-theory}) in Section \ref{sec:loc isom}. We finish the proof of Theorem \ref{thm:main result} in Section \ref{sec:proof}. Finally in Section \ref{sec:mfd}, we outline the case of metric spaces which can be coarsely embedded into a simply-connected complete Riemannian manifold with non-positively sectional curvature.


\textbf{Convention.} Throughout the paper, we fix an infinite-dimensional separable Hilbert space $\HH$. Denote by $\K:=\K(\HH)$ the $C^*$-algebra of compact operators on $\HH$, and $\K_1$ its closed unit ball (with respect to the operator norm).

\section{Preliminaries}\label{sec:pre}

We start with some notions from coarse geometry and higher index theory.

\subsection{Notions from coarse geometry}
Here we collect several basic notions.
\begin{defn}\label{defn:metric spaces with bounded geometry}
	Let $(X,d)$ be a metric space.
	\begin{enumerate}
    \item A subset $A\subseteq X$ is said to be \emph{bounded} if $\sup\{d(x,y): x,y\in A\}$ is finite.
     \item For $x\in X$ and $R>0$, the \emph{open $R$-ball of $x$} is $B(x,R):=\{y\in X: d(x,y)< R \}$.  
    \item $(X,d)$ is said to be \emph{proper} if any bounded closed subset in $X$ is compact.	
    \item For a subset $A\subseteq X$ and $R\geq 0$, the \emph{$R$-neighbourhood of $A$} is defined to be $\Nd_R(A):=\{x\in X: d_X(x,A) \leq R\}$. 
    \item A subset $A\subseteq X$ is called a \emph{net} in $X$ if there exists $C>0$ such that for any $x\in X$ there exists $y\in A$ with $d(x,y) \leq C$. In this case, we also say that $A$ is a \emph{$C$-net} in $X$.
    \item A subset $A\subseteq X$ is said to be \emph{uniformly discrete} if there exists $C>0$ such that $d(x,y) \geq C$ for any $x\neq y$ in $X$. In this case, we also say that $A$ is \emph{$C$-discrete}.
	\item $(X,d)$ is said to have \emph{bounded geometry} if $X$ contains a uniformly discrete net $\Gamma$ with bounded geometry, \emph{i.e.}, for any $ R\geq 0$ there exists an $ N\in\N $ such that $ \#\big(B(x,R)\cap \Gamma \big)\leq N $ for any $ x\in \Gamma$. 
    \end{enumerate}
\end{defn}


\begin{defn}\label{defn:coarse map}
    Let $ f $ be a map between two metric spaces $ (X,d_X) $ and $ (Y,d_Y) $.
	\begin{enumerate}
		\item $f$ is \emph{uniformly expansive} if there exists a non-decreasing function $ \rho_+: [0,\infty) \to [0,\infty)$ such that for any $x,y\in X$ we have:
		\begin{center}
		$d_Y(f(x),f(y)) \leq \rho_+(d_X(x,y))$.
		\end{center}
		\item $f$ is \emph{proper} if for any bounded $B\subseteq Y$, the preimage $f^{-1}(B)$ is bounded in $X$.
		\item $f$ is \emph{coarse} if it is uniformly expansive and proper.
		\item $f$ is \emph{effectively proper} if there exists a proper non-decreasing function $\rho_-: [0,\infty) \to [0,\infty)$ such that for any $x,y\in X$ we have:
		\begin{center}
		$ \rho_-(d_X(x,y)) \leq d_Y(f(x),f(y)) $.
		\end{center}
		\item $f$ is a \emph{coarse embedding} if it is uniformly expansive and effectively proper.	
    \end{enumerate}		
\end{defn}

\begin{defn}\label{defn:coarse equivalence}
	Let $(X,d_X)$ and $(Y,d_Y)$ be metric spaces. 
	\begin{enumerate}
		\item Two maps $f,g: (X,d_X) \to (Y,d_Y)$ are \emph{close} if there exists $R\ge 0$ such that for any $x\in X$, we have $d_Y(f(x),g(x))\le R$.
		\item A coarse map $f: (X,d_X) \to (Y,d_Y)$ is called a \emph{coarse equivalence} if there exists another coarse map $g: (Y,d_Y) \to (X,d_X)$ such that $f\circ g$ and $g\circ f$ are close to identities, where $g$ is called a \emph{coarse inverse} of $f$. 
		\item $(X,d_X)$ and $(Y,d_Y)$ are said to be \emph{coarsely equivalent} if there exists a coarse equivalence from $X$ to $Y$. 
	\end{enumerate}
\end{defn}

\subsection{Roe algbras and quasi-local algebras}

For a proper metric space $(X,d_X)$, recall that an \emph{$X$-module} is a non-degenerate $\ast$-representation $C_0(X) \to \B(\H_X)$ where $\H_X$ is some infinite-dimensional separable Hilbert space. We also say that $\H_X$ is an $X$-module if the representation is clear from the context. An $X$-module is called \emph{ample} if no non-zero element of $C_0(X)$ acts as a compact operator on $\H_X$. Note that every proper metric space $X$ admits an ample $X$-module. 

Let $\H_X$ and $\H_Y$ be ample modules of proper metric spaces $X$ and $Y$, respectively. Given an operator $T\in \B(\H_X,\H_Y)$, the \emph{support} of $T$ is defined to be
\[
\supp(T):=\big\{(y,x)\in Y\times X: \chi_V T\chi_U \neq 0 \mbox{~for~all~neighbourhoods~}U\mbox{~of~}x\mbox{~and~}V\mbox{~of~}y\big\}.
\]
When $X=Y$, the \emph{propagation} of $T \in \B(\H_X)$ is defined to be
\[
\pro(T):=\sup\{d_X(y,x): (y,x)\in \supp(T)\}.
\]
We say that an operator $T\in \B(\H_X)$ has \emph{finite propagation} if $\pro(T)$ is finite, and $T$ is \emph{locally compact} if $fT$ and $Tf$ are compact opretors for all $f\in C_0(X)$.

\begin{defn}\label{defn:Roe algebras}
	For a proper metric space $X$ and an ample $X$-module $\H_X$, $\C[\H_X]$ is defined to be the $*$-algebra of locally compact and finite propagation operators on $\H_X$, and the \emph{Roe algebra $\c(\H_X)$} of $\H_X$ is defined to be the norm-closure of $\C[\H_X]$ in $\B(\H_X)$. 
\end{defn}

It is a standard result that the Roe algebra $\c(\H_X)$ does not depend on the chosen ample module $\H_X$ up to $*$-isomorphisms, hence denoted by $\c(X)$ and called the \emph{Roe algebra of $X$}. Furthermore, $\c(X)$ is a coarse invariant of the metric space $X$ (up to non-canonical $*$-isomorphisms), and their $K$-theories are coarse invariants up to canonical isomorphisms (see, \emph{e.g.}, \cite{roe1993coarse}).

Now we move on to the notion of quasi-locality, introduced by Roe in \cite{roe1988:index-thm-on-open-mfds}.

\begin{defn}\label{defn:quasi-locality}
Let $(X,d)$ be a metric space, and $ \H_X $ be an ample $X$-module. An operator $T\in\B(\H_X)$ is said to be \emph{quasi-local} if for any $\epsilon>0$ there exists $R>0$ such that for any $ A,B \subseteq X $ with $ d(A,B)>R $, we have $ \|\chi_AT\chi_B\|<\epsilon $. 
\end{defn}

It is clear that the set of all quasi-local operators on $\H_X$ forms a $C^*$-subalgebra in $\B(\H_X)$, which leads to the following:

\begin{defn}\label{defn:quasi-local algebra}
For a proper metric space $X$ and an ample $X$-module $\H_X$, the set of locally compact quasi-local operators on $\H_X$ is called the \emph{quasi-local algebra of $\H_X$}, denoted by $C^*_q(\H_X)$. 
\end{defn}

It follows directly from definitions that finite propagation operators are quasi-local, and hence the Roe algebra $C^*(\H_X)$ is a subalgebra of $C^*_q(\H_X)$. For the converse, it was shown in \cite{SZ20} that these two algebras coincide when $X$ has property A. However, the general case is still widely open.

As in the case of Roe algebras, it was proved in \cite[Corollary 2.10]{quasi-local} that the quasi-local algebra $\c_q(\H_X)$ does not depend on the chosen ample module $\H_X$ up to $*$-isomorphisms. Hence we call it the \emph{quasi-local algebra of $X$} and denote by $\c_q(X)$. Furthermore, it was shown that quasi-local algebras are coarse invariants (up to non-canonical $*$-isomorphisms), and their K-theories are coarse invariants up to canonical isomorphisms.

We would like to recall the following characterisation for quasi-locality from \cite{ST19}, which is also the motivation to introduce strong quasi-locality in \cite{quasi-local}. To state the result, we need some more notions:

\begin{defn}\label{defn:variation}
Let $(X,d_X)$ be a metric space and $g:X\rightarrow \C$ be a Borel function. We say that $g$ has \emph{$(\varepsilon,R)$-variation} if for any $x,y\in X$ with $d_X(x,y) < R$, we have $|g(x)-g(y)|< \varepsilon$. We say that $g$ is \emph{bounded} if its norm $\|g\|_\infty:=\sup_{x\in X}|g(x)|$ is finite. 
\end{defn}

\begin{prop}\label{prop:pictures of quasi-locality}
	Let $X$ be a proper metric space, $\H_X$ an ample $X$-module and $T\in\B(\H_X)$ be a locally compact operator. Then the following are equivalent:
	\begin{enumerate}
		\item $T$ is quasi-local in the sense of Definition \ref{defn:quasi-locality};
		\item For any $\varepsilon>0$, there exist $\delta,R>0$ such that for any continuous function $g:X \to \C$ with norm $1$ and $(\delta,R)$-variation then $\|[T,g]\|<\varepsilon$;
		\item For any $\varepsilon>0$, there exist $\delta,R>0$ such that for any Borel function $g:X \to \C$ with norm $1$ and $(\delta,R)$-variation then $\|[T,g]\|<\varepsilon$.
	\end{enumerate}
\end{prop}

Note that the equivalence between (1) and (2) is the ``easier'' part of \cite[Theorem 2.8]{ST19} (see also \cite[Proposition 3.3]{quasi-local}). And also note the same argument can deduce the equivalence between (1) and (3), hence omitted. 

We remark that condition (3) above is the starting point to introduce the notion of strong quasi-locality for general proper metric spaces in Section \ref{sec:sq algebras}.

%

\subsection{Coarse $K$-homology and the assembly map}

Recall that the locally finite $K$-homology groups $ K_i(X)$ ($i=0,1$) for a proper metric space $ X $ are generated by certain cycles modulo certain equivalence relations \cite{Kas75,Kas88}:
	\begin{enumerate}
		\item a cycle for $ K_0(X) $ is a pair $ (\H_X,F) $, where $ \H_X $ is an $ X $-modele and $ F $ is a bounded linear operator acting on $ \H_X $ such that $ F^*F-I $ and $ FF^*-I $ are locally compact, and $ \phi F-F\phi $ is compact for all $ \phi\in C_0(X) $;
		\item a cycle for $ K_1(X) $ is a pair $ (\H_X,F) $, where $ \H_X $ is an $ X $-modele and $ F $ is a self-adjoint operator acting on $ \H_X $ such that $ F^2-I $ are locally compact, and $ \phi F-F\phi $ is compact for all $ \phi\in C_0(X) $;
	\end{enumerate}

In both cases, the equivalence relations on cycles are given by homotopy of the operators $ F $, unitary equivalence, and direct sum with ``degenerate'' cycles, \emph{i.e.}, those cycles for which $ F\phi-\phi F,\phi(F^*F-I) $ and so on, are actually zero.

Now we recall the definition of the assembly map $\mu: K_\ast(X) \to K_\ast(C^*(X))$.

\begin{defn}\label{defn:X-assembly map}
	Let $ (\H_X,F) $ represent a cycle in $ K_0(X) $. For each $ R>0 $, we take a locally finite uniformly bounded open cover $ {\{U_i\}}_i $ of $ X $ with $ \diam(U_i)<R $ for all $ i $, and a continuous partition unity $ {\{\phi_i\}}_i $ subordinate to $ {\{U_i\}}_i $. Define $ F=\sum_i \phi^{1/2}_i T \phi^{1/2}_i$, where the sum converges in the strong operator topology. Then $ (\H_X,F) $ and $ (\H_X,T) $ are equivalent in $ K_0(X) $. The operator $F$ has finite prapagation less than $ R $, hence $ F^*F-I $ and $ FF^*-I $ are in $ \C[X] $. We set:
	\begin{center}
		$ W=\begin{pmatrix}
			I & F \\
			0 & F
		\end{pmatrix}
		\begin{pmatrix}
			I & 0 \\
			-F^* & I
		\end{pmatrix} 
		\begin{pmatrix}
			I & F \\
			0 & F
		\end{pmatrix}
		\begin{pmatrix}
			0 & -I \\
			I & 0
		\end{pmatrix}\in\B(\H_X\oplus\H_X)
		$.
	\end{center}
	Then both $ W $ and $ W^{-1} $ have finite propagations (at most $ 3R $), and
	\begin{center}
		$ W\begin{pmatrix}
			I & 0 \\
			0 & 0
		\end{pmatrix}W^{-1}-
		\begin{pmatrix}
			I & 0 \\
			0 & 0
		\end{pmatrix}\in\C[X]\otimes M_2 $.
	\end{center}
	We then define
	\begin{center}
		$ \mu\big([(\H_X,T)]\big):=\Bigg[W
		\begin{pmatrix}
			I & 0 \\
			0 & 0
		\end{pmatrix}W^{-1}\Bigg]-
		\Bigg[
		\begin{pmatrix}
			I & 0 \\
			0 & 0
		\end{pmatrix}\Bigg] $
	\end{center}
	in $ K_0(\C[X]) $. Furthermore, $ \mu\big([(\H_X,T)]\big) $	defines an element in $ K_0(\c(X)) $ by considering $ \C[X] $ as a $ \ast $-subalgebra of $ \c(X) $. This element is denoted by $ \mu\big([(\H_X,T)]\big)\in K_0(\c(X)) $. Thus, we obtain the assembly map $ \mu:K_0(X)\to K_0(\c(X)) $. Similarly, we can define $ \mu:K_1(X)\to K_1(\c(X)) $.
\end{defn}

To recall the definition of coarse $ K $-homology, we need the notion of \emph{Rips complex}.

\begin{defn}\label{defn:rips complex}
	Let $ X $ be a discrete metric space with bounded geometry. For each $ d\ge0 $, the \emph{Rips complex $ P_d(X) $ at scale $ d $} is defined to be the simplicial complex in which the vertex set is $ X $ and a finite subset $ \{x_0,x_1,\dots,x_q\}\subseteq X $ spans a simplex if and only if $ d(x_i,x_j)\le d $ for all $ 0\le i,j\le q $.
	
	We endow $ P_d(X) $ with the following \emph{spherical metric}. On each connected component of $ P_d(X) $, the spherical metric is the maximal metric whose restriction on each simplex $ \Delta:=\{\sum_{i=0}^q t_ix_i : t_i\ge0,\sum_{i=0}^q t_i=1\} $ is the metric obtained by identifying $\Delta$ with $ S_+^q $ via the map
\[
\sum_{i=0}^q t_ix_i \mapsto \Bigg(\frac{t_0}{\sqrt{\sum_{i=0}^qt_i^2}},\frac{t_1}{\sqrt{\sum_{i=0}^qt_i^2}},\dots,\frac{t_q}{\sqrt{\sum_{i=0}^qt_i^2}}\Bigg),
\]
	where $ S_+^q:=\{(s_0,s_1,\dots,s_q)\in\R^{q+1} : s_i\ge0,\sum_{i=0}^q s_i^2=1\} $ is endowed with the standard Riemannian metric. If $ y_0,y_1 $ belong to two different connected components $ Y_0,Y_1 $ of $ P_d(X) $, respectively, we define
\[
d(y_0,y_1)=\min\big\{d(y_0,x_0)+d_X(x_0,x_1)+d(x_1,y_1) : x_0\in X\cap Y_0,x_1\in X\cap Y_1\big\}.
\]
\end{defn}

Note that for any $ d\ge0 $, $ P_d(X) $ is coarsely equivalent to $ X $ via the inclusion map. If $ d<d' $, then $ P_d(X) \subseteq P_{d'}(X) $ as a subcomplex. Hence we define:

\begin{defn}\label{defn:coarse K-homology}
The coarse $K$-homology for a metric space $(X,d)$, denoted by $ KX_*(X) $, is defined to be the direct limit of the $ K $-homology of $ P_d(X) $, \emph{i.e.}, 
\[
KX_*(X):=\lim_{d\to\infty} K_*(P_d(X)).
\]
Passing the assembly map $\mu$ to direct limit as well, we obtain the \emph{coarse assembly map}:
\begin{equation}\label{EQ:coarse assembly}
\mu:\lim_{d\to\infty} K_*(P_d(X)) \longrightarrow \lim_{d\to\infty} K_*(\c(P_d(X)))\cong K_*(\c(X)).
\end{equation}
\end{defn}



\noindent\textbf{The coarse Baum-Connes conjecture.} If $X$ is a discrete metric space with bounded geometry, then the coarse assembly map $\mu$ defined in (\ref{EQ:coarse assembly}) is bijective.

\noindent\textbf{The coarse Novikov conjecture.} If $X$ is a discrete metric space with bounded geometry, then the coarse assembly map $\mu$ defined in (\ref{EQ:coarse assembly}) is injective.


The coarse Novikov conjecture provides a method to determine nonvanishing of the higher index of the Dirac operator on a noncompact complete Riemannian manifold, which has many significant applications in geometry and topology. In particular, it implies the Gromov positive scalar curvature conjecture and the zero-in-the spectrum conjecture (see \cite{HRY93, roe1993coarse, Roe96, Yu95, Yu00}).

\section{Strongly quasi-local algebras}\label{sec:sq algebras}

In \cite{quasi-local}, the authors introduced a new class of operator algebras for discrete metric spaces, called strongly quasi-local algebras. They sit between Roe algebras and quasi-local algebras, and are shown to be coarse invariants. In this section, we introduce strongly quasi-local algebras for general proper metric spaces and study their coarse geometric properties. This is the preparation to explore the strongly quasi-local perspective on coarse $K$-homology in the next section.


\subsection{Strong quasi-locality}

First, let us recall the notion of strong quasi-locality for operators over discrete metric spaces from \cite{quasi-local}. Recall that $\HH$ is the infinite-dimensional separable Hilbert space fixed at the very beginning, $\K=\K(\HH)$ is the $C^*$-algebra of compact operators on $\HH$, and $\K_1$ is the closed unit ball of $\K$.

Let $X$ be a proper \emph{discrete} metric space and $\H_X$ be an ample $X$-module. For each $x\in X$, denote $\H_x:= \chi_{ \{x\} } \H_X$. An operator $S\in \B(\H_X \otimes \HH)$ can be regarded as an $X$-by-$X$ matrix $(S_{xy})_{x,y\in X}$, where $S_{xy} \in \B(\H_{y} \otimes \HH , \H_{x} \otimes \HH)$. 

Recall that a map $g:X \rightarrow \K$ has \emph{$(\varepsilon,R)$-variation} if for any $x,y\in X$ with $d(x,y) < R$, we have $\|g(x)-g(y)\|< \varepsilon$. The map $g$ is \emph{bounded} if $\|g\|_\infty:=\sup_{x\in X}\|g(x)\|<\infty$. Given a bounded map $g:X \rightarrow \K$, we define an operator $\Lambda(g)\in \B(\H_X \otimes \HH)$ by setting its matrix entry as follows:
\begin{eqnarray}\label{EQ:Lambda discrete}
	\Lambda(g)_{xy}:=
	\begin{cases}
		~\Id_{\H_x} \otimes g(x), & y=x; \\ 
		~0, & \mbox{otherwise}.
	\end{cases}
\end{eqnarray}
Note that this is a block-diagonal operator with respect to the decomposition $\H_X=\bigoplus_{x\in X} \H_x$. 


\begin{defn}[{\cite[Definition 3.4]{quasi-local}}]\label{defn:strongly quasi-local algebra discrete case}
	Let $X$ be a proper \emph{discrete} metric space and $\H_X$ be an ample $X$-module. An operator $T\in \B(\H_X)$ is called \emph{strongly quasi-local} if for any $\varepsilon > 0$ there exist $\delta, R >0$ such that for any map $g:X \rightarrow \K_1$ with $(\delta, R)$-variation, we have $\| [T\otimes \Id_{\HH} , \Lambda(g) ] \|_{ \B (\H_X \otimes \HH)  }< \varepsilon$. The set of all locally compact strongly quasi-local operators on $\H_X$ forms a $C^*$-algebra, called the \emph{strongly quasi-local algebra of $\H_X$} and denoted by $C^*_{sq}(\H_X)$.
\end{defn}

It was shown in \cite{quasi-local} that $C^*_{sq}(\H_X)$ does not depend on the chosen module $\H_X$. Note that when $\H_X$ is chosen to be $\ell^2(X;\HH)$, then the above definition coincides with Definition \ref{introdefn:strong quasi-locality}.

Now we move to the case of proper metric spaces which are \emph{not} necessarily discrete. Our idea is inspired by Proposition \ref{prop:pictures of quasi-locality}(3) while we need the following notion of measurability for operator-valued maps, which can be found in standard textbooks on functional analysis (\emph{e.g.}, \cite[Section V.4]{Yos80}).

\begin{defn}\label{defn:measurable vector-valued function}
	Let $(X,d)$ be a proper metric space with the Borel $\sigma$-algeba and $g: X \to \K$ be a map. 
	\begin{enumerate}
		\item $g$ is said to be \emph{countably-valued} if there is a countable Borel partition $\{E_n\}_{n\in \N}$ of $X$ such that $ g|_{E_n}$ is constant, \emph{i.e.}, $g=\sum_{n\in \N} a_n \chi_{E_n}$ where $a_n \in \K$;
		\item $g$ is said to be \emph{strongly measurable}\footnote{Note that the classic notion of strong measurability is defined for maps on a measure space, where the word ``pointwise'' is replaced by ``almost everywhere''.} if there is a sequence of countably-valued maps $ {\{g_n\}}_{n\in\N} $ such that $g_n \to g$ pointwise;
		\item $g$ is said to be \emph{weakly measurable} if for any bounded linear functional $ \phi: \K \to \C$, the composition $\phi\circ g: X \to \C$ is a Borel function on $X$.
	\end{enumerate}
\end{defn}

The following result is essentially due to Pettis: 

\begin{prop}[\cite{Pet38}]\label{prop:Pettis thm}
	Let $(X,d)$ be a proper metric space with the Borel $\sigma$-algeba and $g: X \to \K$ be a map. Then the following are equivalent:
	\begin{enumerate}
		\item $g$ is strongly measurable;
		\item $g$ is weakly measurable;
		\item there is a sequence of countably-valued maps $ {\{g_n\}}_{n\in\N} $ such that $\|g_n - g\|_\infty \to 0$.
	\end{enumerate}
\end{prop}

A key ingredient for strong quasi-locality is to define the operator $\Lambda(g)$ for ``well-behaved'' operator-valued maps $g: X \to \K$ as in (\ref{EQ:Lambda discrete}). Let us start with the case of countably-valued maps.


\begin{defn}\label{defn:definition of Lambda countably valued}
	Let $(X,d)$ be a proper metric space and $\H_X$ an ample $X$-module. Let $g: X \to \K$ be a bounded countably-valued map, \emph{i.e.}, $g=\sum_{n\in \N} a_n \chi_{E_n}$ where $\{E_n\}_n$ is a Borel partition of $X$ and $a_n \in \K$ with $\sup_{n\in \N} \|a_n\| < \infty$. We decompose $\H_X = \bigoplus_{n} \chi_{E_n} \H_X$ and define $\Lambda(g) \in \prod_{n\in \N}\B\big((\chi_{E_n} \H_X) \otimes \HH\big) \subset \B(\H_X \otimes \HH)$ by
	\[
	\Lambda(g) := (\mathrm{SOT})-\sum_{n\in \N} \chi_{E_n} \otimes a_n.
	\]
	In other words, $\Lambda(g)$ is block-diagonal with respect to the decomposition $\H_X \otimes \HH = \bigoplus_{n} \big( (\chi_{E_n} \H_X) \otimes \HH \big)$ and coincides with $\Id_{\chi_{E_n} \H_X} \otimes a_n$ on each block.
\end{defn}

Now we consider a bounded strongly measurable map $g: X \to \K$. By Definition \ref{defn:measurable vector-valued function} and Proposition \ref{prop:Pettis thm}, there exists a sequence of countably-valued maps $g_n$ such that $\|g_n - g\|_\infty \to 0$. Note that for each $n,m \in \N$, the map $g_n - g_m$ is again countably-valued, hence we have $\|\Lambda(g_n) - \Lambda(g_m)\| = \|g_n - g_m\|_\infty$. Therefore the sequence $\{\Lambda(g_n)\}_{n\in \N}$ is a Cauchy sequence, which converges to some $ T\in\B(\H_X\otimes\HH) $. By a standard argument, it is easy to see that the operator $T$ is independent of the approximating sequence $\{g_n\}_{n\in \N}$. This leads to the following:

\begin{defn}\label{defn:definition of Lambda general case}
	Let $(X,d)$ be a proper metric space and $\H_X$ an ample $X$-module. Let $g: X \to \K$ be a bounded strongly measurable map, and take a sequence of countably-valued maps $\{g_n\}_{n\in \N}$ such that $\|g_n - g\|_\infty \to 0$. We define $\Lambda(g)\in\B(\H_X\otimes\H_0)$ by $\Lambda(g):=\lim_{n\to\infty}\Lambda(g_n)$ (which is well-defined due to the above explanation). Note that we also have $\|\Lambda(g)\| = \|g\|_\infty$.
\end{defn}



%
%

Now we are in the position to introduce strong quasi-locality for operators on proper metric spaces:


\begin{defn}\label{defn:strong quasi-locality}
	Let $(X,d)$ be a proper metric space and $ \H_X $ be an ample $X$-module. An operator $ T\in\B(\H_X) $ is said to be \emph{strongly quasi-local} if for any $ \epsilon>0 $, there exist $\delta,R>0$ such that for any strongly measurable map $ g: X \to \K_1 $ with $ (\delta,R) $-variation (where $\K_1$ denotes the closed unit ball of $\K(\HH)$), we have:
	\[ 
	\big\|[T\otimes\Id_{\HH},\Lambda(g)]\big\|_{\B(\H_X \otimes \HH)}<\epsilon.
	\] 
	It is easy to see that the set of all locally compact strongly quasi-local operators on $\H_X$ forms a $C^*$-algebra, called the \emph{strongly quasi-local algebra of $\H_X$} and denoted by $C^*_{sq}(\H_X)$.
\end{defn}

%

It is straightforward to see that when $X$ is discrete, the above definition coincides with \cite[Definition 3.4]{quasi-local}. Analogous to Roe and quasi-local algebras, it is also easy to see that the algebra $C^*_{sq}(\H_X)$ is \emph{stable} in the sense that $ \cs(\H_X) \otimes M_n \cong \cs(\H_X) $ for any $n\in \N$.

\subsection{Properties of strongly quasi-local algebras}

Recall that for a proper discrete metric space, it was shown in \cite[Proposition 3.7]{quasi-local} that the strongly quasi-local algebra sits between the Roe algebra and the quasi-local algebra. Here we prove a similar result for the general case.

\begin{prop}\label{prop:relations between Roe, QL and SQL}
Let $X$ be a proper metric space and $\H_X$ be an ample $X$-module. Then we have:
\begin{enumerate}
  \item $C^*_{sq}(\H_X) \subseteq C^*_q(\H_X)$;
  \item If $X$ has bounded geometry, then $C^*(\H_X) \subseteq C^*_{sq}(\H_X)$;
  \item If $X$ has bounded geometry and Property A, then $C^*(\H_X) = C^*_{sq}(\H_X) = C^*_q(\H_X)$.
\end{enumerate}
\end{prop}

\begin{proof}
To see (1), we fix a rank-one projection $p \in \B(\HH)$. For any Borel function $g:X \to \C$ with $\|g\|_\infty =1$, define $\tilde{g}: X \rightarrow \K(\HH)_1$ by $\tilde{g}(x):=g(x)p$. Note that $[T\otimes \Id_{\HH} , \Lambda(\tilde{g})] = [T,g] \otimes p$. Hence the conclusion follows from the definition of strong quasi-locality together with Proposition \ref{prop:pictures of quasi-locality}(3).
 
 It remains to prove (2), since (3) follows from \cite[Theorem 3.3]{SZ20} together with (1) and (2). 	To see (2), we take an operator  $T \in \C[\H_X]$ with propagation $R_0$. Since $X$ has bounded geometry, there exists a Borel partition $\{E_n\}_{n\in \N}$ of $X$ satisfying:
\begin{itemize}
  \item there exists $C>0$ such that $\diam(E_n) \leq C$;
  \item there exists $N \in \N$ such that $\# \{m: d(E_n, E_m) \leq R_0\} \leq N$ for each $n\in \N$.
\end{itemize}
We take an $x_n \in E_n$ for each $n\in \N$.
 
Since $T$ has propagation $R_0$, we know that for each $n\in \N$ there are at most $N$-many $k$'s such that $\chi_{E_{k}} T \chi_{E_n} \neq 0$. We enumerate such $k$'s by $k_{1,n}, k_{2,n}, \cdots, k_{N,n}$. For each $i=1,2,\cdots, N$, we denote
\[
T_i=\sum_{n=1}^\infty \chi_{E_{k_{i,n}}}T\chi_{E_n}.
\]
Then it is clear that $T = \sum_{i=1}^N T_i$. 

Now for any $\epsilon>0$, take $\delta:=\frac{\epsilon}{4N\|T\|}$ and $R:=R_0 + 2C$. Then for any strongly measurable map $g: X \to \K_1$ with $(\delta, R)$-variation, we consider a countably-valued map $\tilde{g}: X \to \K_1$ defined by $\tilde{g}|_{E_n}:= g(x_n)$ for each $n\in \N$. Then $\|g-\tilde{g}\|_\infty \leq \delta$, which implies
\[
 \big\|[T \otimes\Id_\HH,\Lambda(g)]\big\| \leq 2\|T\|\delta + \big\|[T\otimes\Id_\HH,\Lambda(\tilde{g})]\big\|.
\]
On the other hand for each $i=1,2,\cdots,N$, we have
\[
\big\|[T_i\otimes\Id_\HH,\Lambda(\tilde{g})]\big\| = \big\| \sum_{n=1}^\infty (\chi_{E_{k_{i,n}}}T\chi_{E_n}) \otimes (g(x_n) - g(x_{k_{i,n}})) \big\| \leq \sup_{n\in \N} \|\chi_{E_{k_{i,n}}}T\chi_{E_n}\| \cdot \delta \leq \frac{\epsilon}{2N}.
\]
Therefore we obtain:
\[
\big\|[T \otimes\Id_\HH,\Lambda(g)]\big\| \leq \frac{\epsilon}{2} + \sum_{i=1}^N \big\|[T_i\otimes\Id_\HH,\Lambda(\tilde{g})]\big\| \leq \frac{\epsilon}{2} + \frac{\epsilon}{2}  = \epsilon,
\]
which concludes the proof.
\end{proof}

Our next goal is to explore the coarse invariance of strongly quasi-local algebras. Note that a similar result hold in the discrete case as shown in \cite[Section3.3]{quasi-local}. While the proof is similar but more involved, we include a detailed one here for completeness.

\begin{prop}\label{prop:strong coarse invariance}
Let $X,Y$ be proper metric spaces with bounded geometry and $\H_X, \H_Y$ be ample modules for $X$ and $Y$, respectively.  Let $f: X \to Y$ be a coarse map with a covering isometry $V: \H_X\rightarrow \H_Y$. Then $V$ induces the following $\ast$-homomorphism
	\[
	\Ad_V: C^*_{sq}(\H_X)\longrightarrow C^*_{sq}(\H_Y), T\mapsto VTV^*.
	\]
Furthermore, the induced K-thoeretic map $(\Ad_V)_*: K_*(C^*_{sq}(\H_X)) \to K_*(C^*_{sq}(\H_Y))$ does not depend on the choice of the covering isometry $V$, hence denoted by $f_*$.
\end{prop}

\begin{proof}
Without loss of generality, we may assume that the coarse map $ f $ is a Borel map (see \cite[Lemma A.3.12]{willett2020higher}). We only show that $VTV^* \in C^*_{sq}(\H_Y)$ if $T\in C^*_{sq}(\H_X)$. The ``Furthermore'' part follows a similar argument as in the case of Roe algebra, hence omitted. 

First note that $VTV^*$ is locally compact since $T$ is, which follows the same argument as in the discrete case (see \cite[Proposition 2.9]{quasi-local}). To see that $VTV^*$ is strongly quasi-local, we assume that $\supp (V) \subseteq \{(y,x):d_Y(f(x),y)<R_0\}$ for some $R_0>0$. Since $ Y $ has bounded geometry, there exists a Borel partition $\{E_n\}_{n\in \N}$ of $Y$ satisfying:
\begin{itemize}
  \item there exists $C>0$ such that $\diam(E_n) \leq C$;
  \item there exists $N\in \N$ such that $\# \{m: d_Y(E_n, E_m) \leq R_0\} \leq N$ for each $n\in \N$.
\end{itemize}
We take an $\gamma_n \in E_n$ for each $n\in \N$. Note that for each $n,m\in \N$, $\chi_{E_m}V\chi_{f^{-1}(E_n)} \neq 0$ implies that $d(E_m, E_n) \leq R_0$. Hence for each $n\in \N$, there are at most $N$-many $m$'s such that $\chi_{E_m}V\chi_{f^{-1}(E_n)} \neq 0$. Enumerate such $m$ by $m_{1,n}, m_{2,n}, \cdots,m_{N,n}$. Hence $V$ can be decomposed into:
\[
V=W_1+W_2+\cdots+W_N \quad \mbox{where} \quad W_i=\sum_{n\in \N}\chi_{E_{m_{i,n}}}V\chi_{f^{-1}(E_n)}.
\]

Let $ M=\max_{1 \leq i \leq N}\|W_i\| $ and $\rho_+: [0,\infty) \to  [0,\infty)$ be the control function of $f$, \emph{i.e.}, $d_Y(f(x),f(y)) \leq \rho_+(d_X(x,y))$ for any $x,y\in X$.
For any $ \epsilon>0 $, $T$ being strongly quasi-local implies that there exist $\delta',R'$ such that for any strongly measurable function $ h:X\to\K_1 $ with $ (\delta',R') $-variation, we have $ \|[T\otimes\Id_\HH,\Lambda(h)]\|<\frac{\epsilon}{2N^2M^2} $.

We take $ \delta=\min\{\frac{\epsilon}{8\|T\|},\frac{\epsilon}{8M^2N^2\|T\|},\frac{\delta'}{2}\} $ and $ R=2C+R_0+\rho_+(R') $. Let $ g:Y \to \K_1 $ be a strongly measurable map with $ (\delta,R) $-variation. Define a countably-valued map $\tilde{g}: Y \to \K_1$ by setting $\tilde{g}|_{E_n}:=g(\gamma_n)$, then it is clear that $\|g-\tilde{g}\|_\infty \leq \delta$. Hence:
\[
\big\|[VTV^*\otimes\Id_\HH,\Lambda(g)]\big\| \leq 2\|T\|\delta + \big\|[VTV^*\otimes\Id_\HH,\Lambda(\tilde{g})]\big\|.
\]


For each $i=1,2,\cdots,N$, define a subset $D_i\subseteq X$ by
\[
D_i:=\{x\in X: \mbox{~there~exists~}y\in Y\mbox{~such~that~}(y,x) \in \supp(W_i)\}
\]
and a map $t_i: D_i \to Y$ by setting $t_i|_{D_i \cap f^{-1}(E_n)}:=\gamma_{m_{i,n}}$ for each $n\in \N$. 
Define a map  $ \varphi_i: X \to \K_1$ by
\[
 \varphi_i(x)=\begin{cases}
			\tilde{g}(t_i(x)) & \text{if } x\in D_i;\\
			0 & \text{if } x\notin D_i.
		\end{cases}
\]
It follows that 
\[
\Lambda(\tilde{g})(W_i \otimes \Id_{\HH}) = (W_i \otimes \Id_{\HH})\Lambda(\varphi_i) \quad \mbox{and} \quad (W_i^* \otimes \Id_{\HH})\Lambda(g) = \Lambda(\varphi_i)(W_i^*\otimes \Id_{\HH}).
\]
On the other hand, for any $x\in D_i \cap f^{-1}(E_n)$ we have 
\[
d_Y(t_i(x),f(x)) \leq  d_Y(\gamma_{m_{i,n}}, E_n) + C \leq R_0 + 2C.
\]
Hence we obtain
\[
\sup_{x\in D_i}\|\varphi_i(x)-\tilde{g}(f(x))\| \leq \delta,
\]
which implies 
\[
\big\|(W_i\otimes\Id_\HH)\Lambda(\varphi_i-\tilde{g}\circ f)\big\|\le\delta M \quad \mbox{~and~} \quad  \big\|\Lambda(\varphi_i-\tilde{g}\circ f)(W^*_i\otimes\Id_\HH)\big\|\le\delta M.
\]
Therefore, for $i,j\in \{1,2,\cdots,N\}$ we have
\begin{eqnarray*}
&&\big\|[(W_iTW_j^*)\otimes \Id_{\HH}, \Lambda(\tilde{g})] \big\| \\
&=& \big\|\big((W_iT)\otimes \Id_{\HH}\big) \Lambda(\varphi_j) (W_j^*\otimes \Id_{\HH}) - (W_i\otimes \Id_{\HH}) \Lambda(\varphi_i) \big((TW_j^*)\otimes \Id_{\HH}\big)\big\| \\
&\leq & \big\|\big((W_iT)\otimes \Id_{\HH}\big) \Lambda(\tilde{g}\circ f) (W_j^*\otimes \Id_{\HH}) - (W_i\otimes \Id_{\HH}) \Lambda(\tilde{g}\circ f) \big((TW_j^*)\otimes \Id_{\HH}\big)\big\| + 2M^2\|T\|\delta\\
&\leq & \big\|(W_i\otimes \Id_{\HH}) [T\otimes \Id_{\HH}, \Lambda(\tilde{g}\circ f)](W_j^*\otimes \Id_{\HH})\big\| + \frac{\varepsilon}{4N^2}.
\end{eqnarray*}

Note that $ \tilde{g} $ has $ (2\delta,R) $-variation on $ Y $, hence $ \tilde{g}\circ f $ has $ (\delta',R') $-variation. This implies that for $i,j\in \{1,2,\cdots,N\}$, we have
\[ 
\big\|(W_i\otimes\Id_\HH)[T\otimes\Id_\HH,\Lambda(\tilde{g}\circ f)](W_j^*\otimes\Id_\HH)\big\|\le M^2\cdot\frac{\epsilon}{2N^2M^2}=\frac{\epsilon}{2N^2}.
\]
Combining the above together, we obtain
\begin{align*}
\big\|[VTV^*\otimes\Id_\HH,\Lambda(g)]\big\| & \leq \frac{\epsilon}{4} + \big\|[VTV^*\otimes\Id_\HH,\Lambda(\tilde{g})]\big\| \\
& \leq \frac{\epsilon}{4}  + \sum_{i,j=1}^N (\frac{\epsilon}{2N^2} + \frac{\varepsilon}{4N^2}) = \epsilon.
\end{align*}
Hence we conclude the proof.
\end{proof}

As a direct corollary, we obtain:

\begin{cor}\label{cor:strong coarse invariance}
Let $\H_X$ and $\H_Y$ be ample modules for proper metric spaces $X$ and $Y$ of bounded geometry, respectively. If $X$ and $Y$ are coarsely equivalent, then the strongly quasi-local algebra $C^*_{sq}(\H_X)$ is $\ast$-isomorphic to $C^*_{sq}(\H_Y)$. In particular, for a proper metric space $X$ of bounded geometry the strongly quasi-local algebra $C^*_{sq}(\H_X)$ does not depend on the chosen ample $X$-module $\H_X$ up to $\ast$-isomorphisms, hence called \emph{the strongly quasi-local algebra of $X$} and denoted by $C^*_{sq}(X)$.
\end{cor}

\section{Localisation algebras for (strong) quasi-locality}\label{sec:localisation}

In this section, we use the notion of strong quasi-locality for proper metric spaces introduced in Section \ref{sec:sq algebras} to provide an alternative version of $K$-homology. Recall that Yu introduced the localisation algebra for a proper metric space \cite{Yu97} and showed that its $K$-theory is isomorphic to the $K$-homology of the underlying space (see also \cite{QR10}). Our idea is to provide a (strongly) quasi-local version of Yu's localisation algebra and show that they have the same $K$-theory (Theorem \ref{introthm:localisation}).

First let us recall Yu's localisation algebra.


\begin{defn}[\cite{Yu97}]\label{defn:localisation Roe algebra}
Let $(X,d)$ be a proper metric space and $\H_X$ be an ample $X$-module. The \emph{localisation algebra $ \L(\H_X) $} is defined to be the norm closure of the $\ast$-algebra of all bounded uniformly continuous map $T: [0,\infty) \to \c(\H_X) $ satisfying $ \lim_{t\to\infty}\pro(T_t)\to 0$ (here we write $T_t$ instead of $T(t)$), with respect to the supremum norm $\|T\|_\infty:=\sup_{t\in [0,\infty)} \|T_t\|$. It is shown in \cite{Yu97} that $\L(\H_X)$ is independent of the module $\H_X$, hence denoted by $\L(X)$. 
\end{defn}

The evaluation map $e: C^*_L(X) \to C^*(X)$ is defined by $e(g)=g(0)$ for $g\in C^*_L(X)$. One can also define a local assembly map $\mu_L$ \cite{Yu97} from the $K$-homology of $X$ to the $K$-theory of the localisation algebra:
\[
\mu_L: K_\ast(X) \longrightarrow K_\ast(C^*_L(X)).
\]

\begin{prop}[\cite{QR10, Yu97}]\label{prop:loc iso to K-homology}
Let $X$ be a proper metric space. Then the local assembly map $\mu_L$ is an isomorphism.
\end{prop}

Consequently, for a proper metric space $X$ with bounded geometry, we have the following commutative diagram:

 \[\xymatrix{
  &   \lim\limits_{d\to\infty}K_\ast(C^*_L(P_d(X)))  \ar[d]_{e_*} \\
		\lim\limits_{d\to\infty}K_*(P_d(X)) \ar[ru]^-{\mu_L}  \ar[r]^-{\mu} & \lim\limits_{d\to\infty}K_\ast(\c(P_d(X)))  
		}\] 
Therefore, the coarse Novikov conjecture is equivalent to that $e_\ast$ is injective and the coarse Baum-Connes conjecture is equivalent to that $e_\ast$ is bijective.



Now we introduce a (strongly) quasi-local version of Definition \ref{defn:localisation Roe algebra}. 

\begin{defn}\label{defn:localisation quasi-local Roe algebra}
Let $X$ be a proper metric space and $\H_X$ be an ample $X$-module. We define the \emph{localisation quasi-local algebra $ \qL(\H_X) $} to be the $ \c $-algebra consisting of all bounded uniformly continuous map $T: [0,\infty) \to \cq(\H_X)$ satisfying that for any $ \epsilon>0,R>0 $ there is a $ t_0>0$ such that for any $t> t_0$ and Borel subsets $ C,D\subseteq X $ with $ d(C,D)>R$, we have $ \|\chi_CT_t\chi_D\|<\epsilon $.
\end{defn}

The intuition behind the above notion is that the ``extent'' of quasi-locality of $T_t$ gets better and better as $t\to \infty$. Hence it is easy to see that $C^*_L(\H_X) \subseteq C^*_{L,q}(\H_X)$. The following lemma follows from the proof of \cite[Theorem 2.8, ``(i) $\Leftrightarrow$ (ii)'']{ST19} and Proposition \ref{prop:pictures of quasi-locality}, hence we omit the proof.

\begin{lem}\label{lem:char for C_Lsq}
Let $X$ be a proper metric space and $\H_X$ an ample $X$-module. Then a bounded uniformly continuous map $T: [0,\infty) \to \cq(\H_X)$ belongs to $\qL(\H_X)$ \emph{if and only if} for any $ \epsilon>0$ there exists $\delta>0$ such that for any $R>0$ there exists $t_0>0$ satisfying the following: for any $t>t_0$ and Borel function $g: X \to \C$ with norm $1$ and $ (\delta,R) $-variation, we have $\|[T_t,g]\|<\epsilon$.
\end{lem}


The above lemma inspires us to introduce the following:

\begin{defn}\label{defn:localisation strongly quasi-local Roe algebra}
Let $X$ be a proper metric space and $\H_X$ an ample $X$-module. The \emph{localisation strongly quasi-local algebra $ \sL(\H_X) $} is defined to be the $ \c $-algebra consisting of all bounded uniformly continuous map $T:[0,\infty) \to \cs(\H_X) $ satisfying that for any $ \epsilon>0$ there exists $\delta>0$ such that for any $R>0$ there exists $t_0>0$ for any $t>t_0$ and strongly measurable function $ g:X\to {\K}_1 $ with $ (\delta,R) $-variation, we have $\|[T_t \otimes \Id_{\HH},\Lambda(g)]\|<\epsilon$.
\end{defn}

Following similar arguments as in the proof of \cite[Proposition 3.7]{Yu97} with slight changes from those of \cite[Proposition 2.9]{quasi-local} and Proposition \ref{prop:strong coarse invariance}, we can show that both $C^*_{L,q}(\H_X)$ and $C^*_{L,sq}(\H_X)$ do not depend on the ample module $\H_X$, hence denoted by $C^*_{L,q}(X)$ and $C^*_{L,sq}(X)$, respectively. We omit the details here. Moreover, from a similar argument as in the proof of Proposition \ref{prop:relations between Roe, QL and SQL}, we have:

\begin{lem}\label{lem:relations between loc alg}
Let $X$ be a proper metric space, then we have:
\[
C^*_L(X) \subseteq C^*_{L,sq}(X) \subseteq C^*_{L,q}(X).
\]
\end{lem}

Note that the above three localisation algebras are closely related to another algebra, which comes from an ``almost commutant picture'' of $K$-homology introduced in \cite{DMW18}:

\begin{defn}[\cite{DMW18}]\label{defn:auxiliary localisation algebra}
Let $X$ be a proper metric space and $\H_X$ an ample $X$-module. We define $ \CC_L(C_0(X)) $ to be the $C^*$-algebra consisting of all bounded uniformly continuous map $T: [0,\infty) \to \B(\H_X)$ satisfying
	\begin{enumerate}
		\item for any $ t\in[0,\infty)$ and $ f\in C_0(X) $, we have $ T_t f, fT_t \in \K(\H_X) $;
		\item for any $ f\in C_0(X) $, we have $\lim_{t\to \infty} \|[T_t,f]\| =0$.
	\end{enumerate}
\end{defn}


As pointed out in \cite{DMW18}, it is a routine work to check that the $C^*$-algebra $ \CC_L(C_0(X))$ does not depend on the module $\H_X$. The following lemma relates our localisation (strongly) quasi-local algebras with $ \CC_L(C_0(X))$:

\begin{lem}\label{lem:inclusion relationship of localisation algebras}
	$ \L(X)\subseteq \sL(X)\subseteq \qL(X)\subseteq \CC_L(C_0(X)) $.
\end{lem}

\begin{proof}
From Lemma \ref{lem:relations between loc alg}, it remains to show that $ \qL(\H_X)\subseteq \CC_L(C_0(X)) $. Let $ T=(T_t)\in \qL(\H_X) $, it suffices to check that $\lim_{t\to \infty} \|[T_t,f]\| =0$ for any real-valued non-negative function $ f\in C_0(X) $ with $\|f\|_\infty \leq 1$. 
	
Fix such an $f\in C_0(X)$. For any $\epsilon>0$, it follows from Lemma \ref{lem:char for C_Lsq} that there exists $\delta>0$ such that for any $R>0$ there exists $t_0>0$ satisfying the following: for any $t>t_0$ and Borel function $g: X \to \C$ with norm $1$ and $ (\delta,R) $-variation, we have $\|[T_t,g]\|<\epsilon$.

Note that $f$ is uniformly consitnuous, hence for the above $\delta$ there exists an $R>0$ such that $f$ has $ (\delta,R) $-variation. Therefore, the above paragraph provides a $t_0$ such that for any $t>t_0$ we have $\|[T_t,f]\|<\epsilon$. This concludes the proof.
\end{proof}

As a special case of the main result in \cite{DMW18}, we know that $K_\ast(\CC_L(C_0(X))) \cong K_\ast(X)$. This inspires the following theorem, which is the main result of this section and implies Theorem \ref{introthm:localisation}. It shows that the above three localisation algebras have the same $K$-theories, hence any of them can serve as the $K$-homology of $X$ by Proposition \ref{prop:loc iso to K-homology}.

\begin{thm}\label{thm:loc alg have same K-theory}
Let $X$ be a proper metric space, then the inclusion maps induce isomorphisms on $K$-theories:
\[
K_\ast(C^*_L(X)) \cong K_\ast(C^*_{L,sq}(X)) \cong K_\ast(C^*_{L,q}(X)) \cong K_\ast(\CC_L(C_0(X))).
\]
\end{thm}

To achieve, we consider a two-sided $\ast$-ideal $ \CC_{L,0}(C_0(X)) $ in $\CC_L(C_0(X))$ consisting of maps $T= (T_t)\in  \L(X) $ such that for any compact $ K\subseteq X $, we have:
\[
\lim_{t\to\infty}\chi_KT_t=0 \quad \mbox{and} \quad   \lim_{t\to\infty}T_t\chi_K=0.
\]
Moreover, we consider $ C^*_{L,0}(X):=\CC_{L,0}(C_0(X))\cap\L(X) $, $ C^*_{L,sq,0}(X):=\CC_{L,0}(C_0(X))\cap\sL(X) $ and $ C^*_{L,q,0}(X):=\CC_{L,0}(C_0(X))\cap\qL(X) $. They are two-sided $\ast$-ideals in $\L(X), \sL(X)$ and $\qL(X)$, respectively.

%



The following is a crucial step to achieve Theorem \ref{thm:loc alg have same K-theory}:

\begin{prop}\label{prop:algebraic iso between loc alg}
Let $X$ be a proper metric space. Then the inclusions from Lemma \ref{lem:inclusion relationship of localisation algebras} induces the following isomorphisms:
\[
\frac{\L(X)}{C^*_{L,0}(X)} \cong \frac{\sL(X)}{C^*_{L,sq,0}(X)} \cong \frac{\qL(X)}{C^*_{L,q,0}(X)} \cong \frac{\CC_L(C_0(X))}{\CC_{L,0}(C_0(X))}.
\]
\end{prop}

%

\begin{proof}
It is clear that the inclusions induce the following injections:
\[
\frac{\L(X)}{C^*_{L,0}(X)} \hookrightarrow \frac{\sL(X)}{C^*_{L,sq,0}(X)} \hookrightarrow \frac{\qL(X)}{C^*_{L,q,0}(X)} \hookrightarrow \frac{\CC_L(C_0(X))}{\CC_{L,0}(C_0(X))}.
\]
Hence it suffices to show that for any $T = (T_t)\in \CC_L(C_0(X)) $, there exists $S = (S_t) \in \L(X)$ such that $T-S\in \CC_{L,0}(C_0(X)) $.
	
For each $ n\in \N $, let $ {\{U_{i}^{(n)}\}}_{i\in \N} $ be a locally finite open cover of $ X $ such that $ \diam(U_{i}^{(n)})<1/n $, and let $ {\{\phi_i^{(n)}\}}_{i\in \N} $ be a partition of unity subordinate to ${\{U_{i}^{(n)}\}}_{i\in \N}$. We define
\[ 
S_{n,t}=\sum_{i\in \N} \sqrt{\phi_i^{(n)}}T_t\sqrt{\phi_i^{(n)}},
\]
where the sum converges in the strong operator topology. It is clear that $ S_{n,t}$ has propagation at most $1/n $.

Since $X$ is proper, we can take a sequence of compact subsets $\{K_m\}_{m\in \N}$ such that each $K_m$ is a closed ball, $K_m \subseteq K_{m+1}$ and $X= \bigcup_{m\in \N} K_m$. For each $m,n \in \N$ we have:
\[
\chi_{K_m}(S_{n,t}-T_t)=\sum\limits_{i\in \N} \chi_{K_m}\sqrt{\phi_i^{(n)}}\big[T_t,\sqrt{\phi_i^{(n)}}\big].
\]
Since $K_m$ is compact, we know that there are only finitely many non-zero items in the above summation. Combining with the assumption that $(T_t) \in \CC_L(C_0(X))$, we obtain that $ \chi_{K_m}(S_{n,t}-T_t) $ tends to zero as $ t  \to \infty$ for each $m,n \in \N$.  Hence for each $ n\in \N $, there is a $ t_n $ such that for any $ t\ge t_n $ we have
\[
\|\chi_{K_n}(S_{n,t}-T_t)\|<1/n \quad \mbox{and} \quad  \|\chi_{K_n}(S_{n+1,t}-T_t)\|<1/n.
\]
	
Set $ t_0=0 $ and we assume that $ t_n $ is monotonically increasing to infinity. Let $ {\{\psi_n\}}_{n\in \N} $ be a partition of unity of $ [0,+\infty) $ such that the $\supp(\psi_n) \subseteq [t_n,t_{n+2}] $. Define
	\[S_t:=\sum\limits_{m=0}^\infty\psi_m(t)S_{m+1,t}\]
for each $t\in [0,\infty)$. It is easy to check that $S:=(S_t)$ belongs to $C^*_L(X)$.

Now we claim that $ T-S \in\CC_{L,0}(C_0(X)) $. Fixing a compact subset $ K\subseteq X $, there is an $ m_0 \in \N $ such that $ K\subseteq K_n $ for any $ n\ge m_0 $.
Then for any $ n>m_0 $ and $ t\in[t_n,t_{n+1}] $, we have 
	\begin{align*}
		\chi_K(T_t-S_t)&=\chi_K\chi_{K_n}(T_t-S_t)\\
		&=\chi_K\chi_{K_n}\big(\sum\limits_{m=0}^\infty\psi_m(t)T_t-\sum\limits_{m=0}^\infty\psi_m(t)S_{m+1,t}\big)\\
		&=\chi_K\chi_{K_n}\big(\psi_{n-1}(t) \cdot (T_t-S_{n,t})+\psi_n(t) \cdot (T_t-S_{n+1,t})\big).
	\end{align*}
	According to the choice of $ t_n $, we have $ \|\chi_K(T_t-S_t)\|\le2/n $. Similarly, we obtain that $ \|(T_t-S_t)\chi_K\|\le2/n $. Hence we conclude the proof.
\end{proof}

We also need the following lemma:

\begin{lem}\label{lem:K-homology lemma 2}
Let $X$ be a proper metric space, and $\H_X$ be an ample $X$-module. Then we have:
\[
K_*(C^*_{L,0}(\H_X))=K_*(C^*_{L,sq,0}(\H_X))=K_*(C^*_{L,q,0}(\H_X))= K_*(\CC_{L,0}(C_0(X))) =0.
\]
\end{lem}

\begin{proof}
The proof follows from a standard Eilenberg Swindle argument (see, \emph{e.g.}, \cite[Lemma 6.4.11]{willett2020higher}), hence omitted.
\end{proof}

\begin{proof}[Proof of Theorem \ref{thm:loc alg have same K-theory}]
The proof follows easily from Proposition \ref{prop:algebraic iso between loc alg}, Lemma \ref{lem:K-homology lemma 2} together with associated six-term short exact sequences. We omit the details.
\end{proof}

%

Now we consider the following commutative diagram:
 \[\xymatrix{
	&   \lim\limits_{d\to\infty}K_\ast(C^*_L(P_d(X)))  \ar[r]^-{i_*} \ar[d]_{e_*} & \lim\limits_{d\to\infty}K_\ast(\sL(P_d(X))) \ar[d]^{{(e_{sq})}_*}\\
	\lim\limits_{d\to\infty}K_*(P_d(X)) \ar[ru]^-{\mu_L}  \ar[r]^-{\mu} & \lim\limits_{d\to\infty}K_\ast(\c(P_d(X)))  \ar[r]^-{i_*} & \lim\limits_{d\to\infty}K_\ast(\cs(P_d(X))) ,}\] 
where $i_\ast$ are induced by inclusions. Recall that the strongly quasi-local coarse assembly map $\mu_{sq}$ is defined to be the composition
\[
\mu_{sq}: \lim_{d\to\infty}K_*(P_d(X)) \stackrel{\mu}{\longrightarrow} \lim_{d\to\infty}K_\ast(\c(P_d(X))) \stackrel{i_\ast}{\longrightarrow} \lim_{d\to\infty}K_\ast(\cs(P_d(X))) \cong K_*(\cs(X)),
\]
where the last isomorphism is from Corollary \ref{cor:strong coarse invariance}. The strongly quasi-local coarse Novikov/Baum-Connes conjecture asserts that $\mu_{sq}$ is injective/bijective.
%
%
%

From Proposition \ref{prop:loc iso to K-homology} and Theorem \ref{thm:loc alg have same K-theory}, we know that the composition 
\[
\mu_{L,sq}:=i_{\ast} \circ \mu_L: \lim\limits_{d\to\infty}K_*(P_d(X)) \longrightarrow \lim\limits_{d\to\infty}K_\ast(C^*_L(P_d(X))) \longrightarrow \lim\limits_{d\to\infty}K_\ast(\sL(P_d(X)))
\]
is an isomorphism. Consequently, we obtain the following:

\begin{cor}\label{cor:equiv for sqcbc}
Let $X$ be a discrete metric space with bounded geometry. Then the strongly quasi-local coarse Baum-Connes/Novikov conjecture holds for $X$ \emph{if and only if} the following homomorphism induced by evaluation at $0$:
\[
{(e_{sq})}_*: \lim\limits_{d\to\infty}K_\ast(\sL(P_d(X))) \longrightarrow \lim\limits_{d\to\infty}K_\ast(\cs(P_d(X)))
\]
is bijective/injective.
\end{cor}

We can similarly define a quasi-local coarse assembly map 
\[
\mu_{L,q}: \lim_{d\to\infty}K_*(P_d(X)) \longrightarrow K_*(C^*_{q}(X))
\]
and consider the corresponding quasi-local coarse Baum-Connes/Novikov conjecture. Moreover, according to Theorem \ref{thm:loc alg have same K-theory} again, we also have an analogous result to Corollary \ref{cor:equiv for sqcbc} for quasi-local algebras.

\section{Twisted Algebras}\label{sec:twisted alg}

The notion of twisted Roe algebras for metric spaces which admit a coarse embedding into a Banach space with Property (H) were introduced in \cite{CWY15}. They play a key role in the proof of the coarse Novikov conjecture for such spaces (\cite[Theorem 1.1]{CWY15}). In this section, we study a modified version of the twisted Roe algebras and introduce their quasi-local counterparts. These algebras will provide a bridge to link the $K$-theories of Roe algebras and strongly quasi-local ones.


Let us fix some notation following \cite{CWY15}. Throughout the section, let $X$ be a discrete metric space with bounded geometry which admits a coarse embedding $f: X \to V$ into a real Banach space $V$ with Property (H). From Definition \ref{defn:CWY15}, there exist increasing sequences of finite dimensional subspaces $ \{V_n\}_{n\in\N} $ of $V$ and $ \{W_n\}_{n\in\N} $ of a real Hilbert space $W$ such that:
\begin{enumerate}
 \item each $V_n$ and $W_n$ are even dimensional, and $\bigcup_{n=1}^\infty V_n$ is dense in $V$;
 \item there exists a uniformly continuous map $ \psi:S(\bigcup_{n=1}^\infty V_n) \rightarrow S(\bigcup_{n=1}^\infty W_n) $ such that $\psi|_{S(V_n)}$ is a homeomorphism onto $ S(W_n) $ for any $n\in\N$. 
\end{enumerate}

By a slight modification, we may assume without loss of generality that 
\[
f(X) \, \subset \, \bigcup_{n=1}^\infty V_n.
\]
For each $d \geq 0$, the coarse embedding $f: X \to \bigcup_{n=1}^\infty V_n$ can be extended to a continuous coarse embedding $f: P_d(X) \to \bigcup_{n=1}^\infty V_n$ by affine extension (see \cite[Section 3.1]{CWY15}). We also take a countable dense subset $Z_d \subset P_d(X)$ such that $Z_d \subseteq Z_{d'}$ whenever $d<d'$.

For each $n\in \N$, let $\Cliff(W_n)$ be the complex Clifford algebra of $ W_n $ with respect to the relation $w^2=\|w\|^2$ for all $w\in W_n$. Denote
\[
  \A_n:=C_0(V_n)\otimes\Cliff(W_n) \cong C_0(V_n,\Cliff(W_n)).
\]
We follow the notation from Section \ref{sec:sq algebras} that $\K=\K(\HH)$ denotes the $C^*$-algebra of compact operators on the fixed infinite dimensional separable Hilbert space $\HH$. For a function $h \in \A_n \otimes \K \cong C_0(V_n,\Cliff(W_n)\otimes \K) $, we define the support of $h$ to be
\[
\supp(h)=\overline{\big\{ x\in V_n: h(x)\neq 0 \big\}}.
\]

Now we consider two $C^*$-algebras which serve as coefficients for twisted algebras. Let $\prod_{n=1}^\infty (\A_n\otimes\K)$ be the $C^*$-algebra direct product and $\bigoplus_{n=1}^\infty (\A_n\otimes\K)$ be the $C^*$-algebra direct sum. Denote the quotient algebra by 
\[
\mathcal{Q}((\mathcal{A}_n\otimes\mathfrak{K})_{n\in \N}):= \frac{\prod_{n=1}^\infty (\A_n\otimes\K)}{\bigoplus_{n=1}^\infty (\A_n\otimes\K)},
\]
where the norm of $[(h_1, \cdots, h_n, \cdots)] \in \mathcal{Q}((\mathcal{A}_n\otimes\mathfrak{K})_{n\in \N})$ can be calculated by:
\[
\big\| [(h_1, \cdots, h_n, \cdots)] \big\| = \limsup_{n\to \infty} \|h_n\|_{\A_n\otimes\K}.
\]
Also denote the $C^*$-algebra
\[
\mathcal{Q}((\A_n)_{n\in \N}) := \frac{\prod_{n=1}^\infty \A_n}{\bigoplus_{n=1}^\infty\A_n}.
\]
Since $\K$ is nuclear, we have:
\begin{equation}\label{EQ:q into Q}
\mathcal{Q}((\A_n)_{n\in \N}) \otimes \K \cong \frac{(\prod_{n=1}^\infty \A_n) \otimes \K}{(\bigoplus_{n=1}^\infty\A_n) \otimes \K} \hookrightarrow \frac{\prod_{n=1}^\infty (\A_n\otimes\K)}{\bigoplus_{n=1}^\infty (\A_n\otimes\K)} = \Q,
\end{equation}
where the second inclusion is induced by the natural inclusion 
\[
\big(\prod_{n=1}^\infty \A_n\big) \otimes \K \hookrightarrow \prod_{n=1}^\infty (\A_n\otimes\K).
\]
Hence in the sequel, we regard $\q$ as a $C^*$-subalgebra of $\Q$ via (\ref{EQ:q into Q}). Also for short we write $\Q$ and $\mathcal{Q}((\A_n)_{n})$ instead of $\mathcal{Q}((\mathcal{A}_n\otimes\mathfrak{K})_{n\in \N})$ and $\mathcal{Q}((\A_n)_{n\in \N})$, respectively.

%



\subsection{Preliminaries on Hilbert modules}\label{ssec:Hilbert modules}

To define twisted algebras, we need to recall some facts from the theory of Hilbert module. Readers can refer to standard textbooks (\emph{e.g.}, \cite{Lan95}) for basic notions and more details. 

Let $Z$ be a countable set and $A$ be a $C^*$-algebra. We denote a map $Z \to A$ which maps $x \mapsto a_x$ by $\sum_{x\in Z} a_x[x]$. Let
\[
\ell^2(Z;A) := \big\{ \sum_{x\in Z} a_x[x] : a_x \in A \mbox{~and~} \sum_{x\in Z}a_x^*a_x \mbox{~converges~in~norm}\big\}.
\]
Then $\ell^2(Z;A)$ is a Hilbert $A$-module induced by:
\[
\big\langle \sum_{x\in Z} a_x[x], \sum_{x\in Z}b_x[x] \big\rangle := \sum_{x\in Z}a_x^*b_x,
\]
\[
\big(\sum_{x\in Z} a_x[x] \big) \cdot a := \sum_{x\in Z}(a_xa)[x]
\]
for any $a\in A$. Let $\B(\ell^2(Z;A))$ be the $C^*$-algebra of all module homomorphisms from $\ell^2(Z;A)$ to itself for which there exists an adjoint module homomorphism, and $\K(\ell^2(Z;A))$ be the norm closure of all finite-rank operators in $\B(\ell^2(Z;A))$. Note that $\K(\ell^2(Z;A))$ is a $\ast$-ideal in $\B(\ell^2(Z;A))$.

Moreover, note that there is a $\ast$-representation $\rho: \ell^\infty(Z) \to \B(\ell^2(Z;A))$ by pointwise scalar multiplication. For abbreviation, we will write $fT$ instead of $\rho(f)T$ for $f\in \ell^\infty(Z)$ and $T\in \B(\ell^2(Z;A))$.

For each $T \in \B(\ell^2(Z;A))$, we can associate a matrix form $(T_{xy})_{x,y\in Z}$ as follows. For $x,y\in Z$, we define a linear map $T_{xy}: A\to A$ by
\[
T_{xy}(a) := \big( T(a[y]) \big)(x) \mbox{~for~} a\in A,
\]
called the \emph{$xy$-matrix entry of $T$}. It is easy to check that $\|T_{xy}\|_{\B(A)}=\|\chi_{\{x\}} T \chi_{\{y\}}\|_{\B(\ell^2(Z;A))}$, and each $T_{xy}$ is indeed in the multiplier algebra $\mathcal{M}(A)$. Furthermore, it follows from $T\in \B(\ell^2(Z;A))$ that $T$ can be recovered by its matrix form via matrix multiplication. More precisely, for each $\sum_{x\in Z} a_x[x] \in \ell^2(Z;A)$ we have:
\begin{enumerate}
 \item for $x\in Z$, the sum $\sum_{y\in Z} T_{xy}(a_y)$ converges in norm;
 \item the map $x \mapsto \sum_{y\in Z} T_{xy}(a_y)$ belongs to $\ell^2(Z;A)$;
 \item $T(\sum_{x\in Z} a_x[x]) = \sum_{x\in Z} (\sum_{y\in Z} T_{xy}(a_y)) [x]$.
\end{enumerate}

%

We need the following auxiliary lemma to define the Bott maps in Section \ref{ssec:sq Bott}:

\begin{lem}\label{lem:the norm on Hilbert module}
Let $ (\pi,\H) $ be a faithful representation of a $C^*$-algebra $A$ (where the Hilbert space $ \H $ is not necessarily separable), and $Z$ be a countable set. For any $T \in \B(\ell^2(Z;\H))$, we write its matrix form as $T=(T_{xy})_{x,y\in Z}$ where $T_{xy} \in \B(\H)$. Assume that each $T_{xy} \in \mathcal{M}(A)$. Define a linear operator $\Psi(T)$ on the Hilbert module $\ell^2(Z;A)$ by:
\begin{equation}\label{EQ:norm control}
\Psi(T)\big(\sum_{x \in Z}a_x [x]\big) := \sum_{x\in Z} \big( \sum_{y\in Z} T_{xy}(a_y) \big)[x]
\end{equation}
for any $\sum_{x \in Z}a_x [x] \in \ell^2(Z;A)$. Then $\Psi(T)$ is well-defined and belongs to $\B(\ell^2(Z;A))$. Moreover, we have $\|\Psi(T)\| \leq \|T\|$ and $T_{xy}=\Psi(T)_{xy}$ for any $x,y\in Z$.
\end{lem}

\begin{proof}
For any finite subset $F \subset Z$ and $\sum_{x\in F} a_x[x] \in \ell^2(X;A)$ with $\|\sum_{x\in F} a_x^*a_x\|=1$, we consider a vector $b\in \ell^2(F,A)$ by $b_x:=\sum_{y\in F}T_{xy}a_y$ for $x\in F$. Then we have:
\[
\big\|\sum_{x\in F}b_x[x]\big\|^2 =\big\|\sum_{x\in F}b_x^*b_x\big\|=\sup\limits_{\|\xi\|=1} \big\langle\sum\limits_{x\in F}b_x^*b_x\xi,\xi \big\rangle=\sup_{\|\xi\|=1}\sum_{x\in F}{\|b_x\xi\|}^2.
\]
For any $\xi \in \H$ with $ \|\xi\|=1 $, we set $ \xi'=\sum_{x\in F}(a_x\xi)[x] \in \ell^2(Z;\H)$. Then:
\[
\|\xi'\|^2=\sum_{x\in F}{\|a_x\xi\|}^2 = \big\langle \big(\sum_{x\in F}a_x^*a_x\big)\xi,\xi \big\rangle \leq \big\|\sum_{x\in F}a_x^*a_x\big\| \cdot \|\xi\|^2=1,
\]
and hence
\[
\sum_{x\in F}{\|b_x\xi\|}^2 = \big\|\sum_{x\in F}(b_x\xi)[x]\big\|^2 =  \|T\xi'\|^2 \leq \|T\|^2 \cdot \|\xi'\|^2 \leq \|T\|^2.
\]
This concludes the proof.
\end{proof}

\subsection{Twisted Roe algebras}

Here we recall the original definition of twisted Roe algebras from \cite[Section 3.1]{CWY15}. We also introduce a restricted version and study their relations. This will play a key role in the sequel to study the Bott map for strongly quasi-local algebras and to prove the main theorem. Let us fix a $d\geq 0$, and recall that $Z_d$ is a countable dense subset in the Rips complex $P_d(X)$. 

\begin{defn}[{\cite[Definition 3.1]{CWY15}}]\label{defn:original algebraic twisted Roe}
Define $ \C[P_d(X),\Q] $ to be the set of bounded functions $T: Z_d \times Z_d \to \Q$ satisfying the following conditions:
\begin{enumerate}
	\item for any bounded subset $B\subseteq X$, the set $\big \{(x,y)\in (B\times B) \cap (Z_d\times Z_d) : T(x,y) \neq 0 \big\}$ is finite;
	\item there exists an $ L>0 $ such that $\#\{x\in Z_d: T(x,y) \neq 0\} < L$ and $\# \{y\in Z_d : T(x,y)\neq 0\} < L $ for each $x,y\in Z_d$;
	\item there exists an $ R\ge 0 $ such that $ T(x,y)=0 $ whenever $ d(x,y)>R $ for $x,y\in Z_d$;
	\item there exists an $ r>0 $ such that for any $ x,y \in Z_d $, $T(x,y)$ is of the form $[(h_1,\cdots,h_n,\cdots)] $ satisfying $ \supp(h_n) \subseteq B_{V_n}(f(x),r)$ for $n \in \N$ large enough such that $ f(x)\in V_n$.
\end{enumerate}	
\end{defn}

The algebraic structure for $\C[P_d(X),\Q]$ is defined by regarding elements $T$ as $Z_d \times Z_d$-matrices. The $\ast$-structure for $\C[P_d(X),\Q]$ is defined by 
\[
(T^*)(x,y):=[(h_1^*, \cdots, h_n^*, \cdots)]
\]
where $T(y,x)=[(h_1, \cdots, h_n, \cdots)]$ for all $x,y\in Z_d$. 

Denote by $E_1$ the Hilbert module $\ell^2(Z_d; \Q)$. It is clear that the $\ast$-algebra $\C[P_d(X),\Q]$ has an action on $E_1$ by the formula
\begin{equation}\label{EQ:action on E1}
T\big(\sum_{x\in Z_d}a_x[x]\big)=\sum_{x\in Z_d}\big(\sum\limits_{y\in Z_d}T(x,y)a_y\big)[x]
\end{equation}
for $T \in \C[P_d(X),\Q]$ and $\sum_{x\in Z_d}a_x[x] \in \ell^2(Z_d; \Q)$. Also note that $T$ is an adjointable module homomorphism, hence defines an element in $\B(E_1)$. It is easy to check that the $xy$-matrix entry of $T\in \B(E_1)$ coincides with $T(x,y)$.

The following is the twisted Roe algebra introduced in \cite[Definition 3.2]{CWY15}:

\begin{defn}[\cite{CWY15}]\label{defn:original twisted Roe algebra}
The \emph{twisted Roe algebra} $C^*(P_d(X),\Q)$ is defined to be the norm closure of $\C[P_d(X),\Q]$ in $ \B(E_1) $.
\end{defn}

Now we introduce a modified version of the twisted Roe algebra. The idea is to restrict the coefficient algebra $\Q$ to the subalgebra $\q$.

\begin{defn}\label{defn:algebraic restricted twisted Roe algebra}
Define $ \C[P_d(X),\q] $ to be the $\ast$-algebra consisting of all bounded functions $T: Z_d \times Z_d \to \q$ satisfying condition (1) $\thicksim$ (4) in Definition \ref{defn:original algebraic twisted Roe}.
\end{defn}

Note that the only difference between the $ \ast $-algebras $ \C[P_d(X),\Q] $ and $ \C[P_d(X),\q] $ is the requirement on ranges of their elements. It is also clear that $ \C[P_d(X),\q] $ is a $\ast$-subalgebra of $ \C[P_d(X),\Q] $.

We denote by $E$ the Hilbert module $\ell^2(Z_d; \q)$. Note that the algebra $\C[P_d(X),\q]$ has an action on $E$ by the same formula (\ref{EQ:action on E1}) and similarly, $\C[P_d(X),\q]$ can be regarded as a subalgebra in $\B(E)$. This leads to the following:

\begin{defn}\label{defn:restricted twisted Roe algebra}
The \emph{restricted twisted Roe algebra} $C^*(P_d(X),\q)$ is defined to be the norm closure of $\C[P_d(X),\q]$ in $\B(E)$.
\end{defn}

\begin{rem}
The motivation to introduce the restricted version of the twisted Roe algebra is to pave the way for the Bott maps constructed in Section \ref{sec:Bott}. We will see that the Bott map can be defined for strongly quasi-local algebras using the smaller coefficient algebra $\q$, while it is unclear to us whether this can be done in terms of the original $\Q$. See Remark \ref{rem:diff between twisted Roe} for more details.
\end{rem}

To study the relation between the two twisted Roe algebras, we need an auxiliary algebra which is a subalgebra in $C^*(P_d(X),\Q)$. Note that the $\ast$-algebra $\C[P_d(X),\q]$ also has an action on the module $E_1$ by the same formula (\ref{EQ:action on E1}), hence we can consider its norm closure in $\B(E_1)$:
\[
\mathfrak{A} :=\overline{\C[P_d(X),\q]}^{\|\cdot\|_{\B(E_1)}}.
\]
It is clear that $\mathfrak{A}$ is a $C^*$-subalgebra of $C^*(P_d(X),\Q)$. The following lemma shows that $\mathfrak{A}$ is isomorphic to the restricted twisted Roe algebra, hence builds a bridge between $C^*(P_d(X),\q)$ and $C^*(P_d(X),\Q)$.

\begin{lem}\label{lem:restricted twisted Roe and A}
The identity map on $\C[P_d(X),\q]$ can be extended continuously to a $\ast$-isomorphism from $C^*(P_d(X),\q)$ to $\mathfrak{A}$. In particular, we have $\|T\|_{\B(E)} = \|T\|_{\B(E_1)}$ for any $T \in \C[P_d(X),\q]$.
\end{lem}

\begin{proof}
To prove the lemma, we consider the restriction map as follows. For any $T \in \mathfrak{A}$, we claim that $T(E) \subseteq E$. In fact, assume that there exists a sequence $\{T_n\}_n$ in $\C[P_d(X),\q]$ converging to $T$ in $\B(E_1)$. It is clear that $T_n(E) \subseteq E$ for each $n$. Hence for any $\xi \in E \subseteq E_1$, we have $T\xi = \lim_{n\to \infty}T_n(\xi) \in E$ since $E$ is closed. This leads to a $C^*$-homomorphism 
\[
\Phi: \mathfrak{A} \to C^*(P_d(X),\q), \quad T \mapsto T|_E.
\]

It is clear that $\Phi$ is surjective since its image contains the dense subalgebra $\C[P_d(X),\q]$, hence it suffices to show that $\Phi$ is injective. Assume that $\Phi(T)=0$ for some $T \in \mathfrak{A}$, and take a sequence $\{T_n\}_n$ in $\C[P_d(X),\q]$ converging to $T$ in $\B(E_1)$. For any $x,y\in Z_d$, it is clear that the matrix entry $(T_n)_{xy}$ converges to $T_{xy}$ in $\mathcal{M}(\Q)$. Since each $(T_n)_{xy} \in \q$, which is regarded as an element in $\mathcal{M}(\Q)$ via the inclusion $i: \q \to \mathcal{M}(\Q)$, we obtain that $T_{xy}$ also belongs to the image of $i$. Hence $T_{xy}(\q)=0$ implies that $T_{xy}=0$, and therefore we obtain that $T=0$. 
\end{proof}

As a direct corollary, we obtain the following:

\begin{cor}\label{cor:relation between twisted Roe}
There is a $C^*$-monomorphism 
\[
\iota: C^*(P_d(X),\q) \to C^*(P_d(X),\Q)
\]
induced by the inclusion $\C[P_d(X),\q] \hookrightarrow \C[P_d(X),\Q]$.
\end{cor}

\subsection{Twisted quasi-local algebras}

Now we introduce a quasi-local version of the twisted algebras. Due to the lack of control on propagations, we need a slightly different approach.

First note that any continuous bounded function $h:V \to \C$ can be regarded as an element in the multiplier algebra $\mathcal{M}(\mathcal{Q}((\A_n)_{n}))$ by 
\[
h\cdot [(h_1,\cdots,h_n,\cdots)] = [(h|_{V_1}\cdot h_1, \cdots, h|_{V_n}\cdot h_n, \cdots)],
\]
where $[(h_1,\cdots,h_n,\cdots)] \in \mathcal{Q}((\A_n)_{n})$. Hence we can also regard $h$ as an element in $\B(E)$ by setting its matrix entry to be:
\[
h_{xy}:=
\begin{cases}
~h \otimes \Id_{\HH} \, (\in \mathcal{M}(\q)), & x=y; \\ 
~0, & \mbox{otherwise}.
\end{cases}
\]

\begin{defn}\label{defn:twisted strongly quasi-local algebra}
The \emph{twisted quasi-local algebra} $C^*_q(P_d(X),\q)$ is defined to be the subset in $\B(E)$ consisting of elements $T$ satisfying the following conditions: 
\begin{enumerate}
     \item each matrix entry of $T$ belongs to $\q$;
     \item for any compact $K \subseteq P_d(X)$, we have that $\chi_K T$ and $T \chi_K$ belongs to $\K(E)$;
     \item for any $\epsilon>0$ there is an $r_1$ such that for any Borel subsets $ C,D\subseteq P_d(X) $ with $ d(C,D)>r_1 $, we have $ \|\chi_CT\chi_D\|< \epsilon $ and $ \|\chi_DT\chi_C\|< \epsilon $;
     \item for any $ \epsilon>0 $, there is an $ r_2 $ such that for any Borel subset $ C\subseteq P_d(X) $ and $h\in C_b(V)_1$ with $ d(f(C),\supp(h))>r_2 $, we have $ \|\chi_CTh\|<\epsilon $.
\end{enumerate}
\end{defn}

\begin{rem}\label{rem:twisted sq alg star condition}
It follows from condition (3) and (4) in Definition \ref{defn:twisted strongly quasi-local algebra} that for any $T \in C^*_q(P_d(X),\q)$ and $\epsilon>0$, there exists $r'_2>0$ such that for any Borel subset $ C\subseteq P_d(X) $ and $h\in C_b(V)_1$ with $ d(f(C),\supp(h))>r'_2 $, we have $ \|hT\chi_C\|<\epsilon $. In fact, assume that $\rho_+: [0,\infty) \to [0,\infty)$ is a proper function such that $\|f(x)-f(y)\| \leq \rho_+ (d(x,y))$ for any $x,y\in P_d(X)$. For the given $\epsilon>0$, there exist $r_1, r_2>0$ satisfying condition (3) and (4) therein for $\epsilon/2$. We take $r'_2:=\rho_+(r_1)+r_2$, then for any Borel subset $ C\subseteq P_d(X) $ and $h\in C_b(V)_1$ with $ d(f(C),\supp(h))>r'_2 $, we have $ d(f(C),\supp(h))>r_2 $ and hence:
\[
\|hT \chi_C\| = \|T\chi_C h\| \leq \|\chi_{\Nd^c_{r_1}(C)}T \chi_C h\| + \|\chi_{\Nd_{r_1}(C)}Th \chi_C \| \leq \epsilon/2 + \epsilon/2 = \epsilon.
\]
In the sequel, we will use this observation without further explanation.
\end{rem}



The following result shows that twisted quasi-local algebras are indeed $C^*$-algebras:

\begin{lem}\label{lem:twisted strongly quasi-local algebra is well defined}
The set $C^*_q(P_d(X),\q)$ forms a $C^*$-subalgebra of $\B(E)$.
\end{lem}

\begin{proof}
It follows from Remark \ref{rem:twisted sq alg star condition} that $C^*_q(P_d(X),\q)$ is a complete $\ast$-closed linear space in $\B(E)$. It remains to show that $C^*_q(P_d(X),\q)$ is closed under multiplication. Let $T_1, T_2 \in C^*_q(P_d(X),\q)$ and fix $x,y\in Z_d$. We claim that the series $\sum_{z\in Z_d} (T_1)_{xz}(T_2)_{zy}$ converges in norm $\|\cdot\|_{\q}$, which implies that condition (1) holds for $T_1T_2$. In fact due to condition (2), for any $\epsilon>0$ there exist finite rank operators $S_1,S_2 \in \K(E)$ with finitely many non-zero matrix entry such that $\|\chi_{\{x\}}T_1 - S_1\| < \epsilon$ and $\|T_2\chi_{\{y\}} - S_2\| < \epsilon$, which implies that $\|\chi_{\{x\}}T_1 - \chi_{\{x\}} S_1\| < \epsilon$ and $\|T_2\chi_{\{y\}} - S_2\chi_{\{y\}}\| < \epsilon$. Hence there exists a finite $F_0 \subset Z_d$ such that for any $F \subseteq Z_d \setminus F_0$, we have
\[
\epsilon > \|\chi_{\{x\}}T_1\chi_F - \chi_{\{x\}} S_1\chi_F\| = \|\chi_{\{x\}}T_1\chi_F\| \quad \mbox{and} \quad \epsilon > \|\chi_F T_2\chi_{\{y\}} - \chi_F S_2\chi_{\{y\}}\| = \|\chi_F T_2\chi_{\{y\}}\|.
\]
Therefore we obtain that 
\[
\big\| \sum_{z\in F} (T_1)_{xz}(T_2)_{zy} \big\|_{\q} = \big\|\chi_{\{x\}}T_1\chi_FT_2\chi_{\{y\}}\big\|_{\B(E)} < \epsilon^2,
\]
which concludes the claim. It is routine to check that condition (2)-(4) hold for $T_1T_2$, hence we finish the proof.
\end{proof}


It is straightforward to check that $C^*(P_d(X),\q)$ is a $C^*$-subalgebra of $C^*_q(P_d(X),\q)$. For later use, we denote the natural inclusion:
\begin{equation}\label{EQ: iA}
i^{\mathcal{A}}: C^*(P_d(X),\q) \hookrightarrow C^*_q(P_d(X),\q).
\end{equation}
In the sequel, we will show that $i^{\mathcal{A}}$ induces isomorphisms on $K$-theories. This is a crucial step to achieve the main result.

\section{Construction of the Bott maps}\label{sec:Bott}

This section is devoted to linking the $K$-theories of Roe and strongly quasi-local algebras with those of their twisted counterparts. To achieve, we construct the so-called Bott maps $\beta_r, \beta_{sq}$ following the idea from \cite{CWY15} and fill them into the following commutative diagram: 
\begin{equation}\label{EQ:Bott map diagram}
	\xymatrix{
		K_\ast(C^*(P_d(X)))  \ar[r]^-{\beta_r} \ar[d]_{i_\ast} & K_\ast(C^*(P_d(X),\q)) \ar[d]^{i^{\mathcal{A}}_\ast} \\
		K_\ast(C^*_{sq}(P_d(X)))  \ar[r]^-{\beta_{sq}} & K_\ast(C^*_{q}(P_d(X),\q)),}
\end{equation}
where the vertical lines are induced by inclusions. We follow the same notation introduced in Section \ref{sec:twisted alg} and fix $d \geq 0$.


\subsection{The Roe algebra case}\label{ssec:Roe Bott}

We start by recalling the Bott map constructed in \cite{CWY15} for Roe algebras. 
A key observation here is that the image of the Bott map  
\begin{center}
	$ \beta:K_\ast(P_d(X))\to K_\ast(\c(P_d(X),\mathcal{Q}((M_2(\mathcal{A}_n^+)\otimes\mathfrak{K})_{n}))) $
\end{center} 
can be made into the $K$-theory of the restricted twisted algebra introduced above, which makes it possible to relate to the strongly quasi-local case.

%
%

The crucial ingredient to construct the Bott map is the existence of the uniformly almost flat Bott generators. For each $n\in \N$, denote by $\A_n^+$ the $C^*$-algebra unitization of $\A_n=C_0(V_n,\Cliff(W_n))$ and by $M_2(\A_n^+)=\A_n^+ \otimes M_2(\C)$ the algebra of $2 \times 2$ matrices over $\A_n^+$. Set $b_0:=\begin{pmatrix}
		0 & 0 \\
		0 & 1
	\end{pmatrix} \in M_2(\A_n^+)$. For a matrix $a=\begin{pmatrix}
		a_{11} & a_{12} \\
		a_{21} & a_{22}
	\end{pmatrix}$
of functions on $V_n$, we define the support of $a$ to be
\[
\supp(a):= \bigcup_{i,j=1}^2 \supp(a_{ij}).
\]

\begin{lem}[{\cite[Lemma 3.6 and Definition 3.7]{CWY15}}]\label{lem:uniformly almost flatness of b}
For any $R>0$ and $\epsilon>0$, there exist $r>0$ and a family of idempotents $ \{b_{x,r}^{(n)}\}_{n\in \N,x\in V_n} $ in $ M_2(\mathcal{A}_n^+) $ which is \emph{$ (R,\epsilon;r) $-flat} in the following sense: $ \supp(b_{x,r}^{(n)}-b_0)\subseteq B_{V_n}(x,r) $ and for any $n\in \N$ and $x,y\in V_n$ with $\|x-y\|<R$, we have 
	\[
	\sup_{v\in V_n}\big\|b_{x,r}^{(n)} - b_{y,r}^{(n)}\big\|_{\Cliff(W_n)\otimes M_2(\C)}<\epsilon.
	\]
\end{lem}

For convenience to readers, we briefly recall the construction here. Details can be found in \cite[Section 3.2]{CWY15}. For each $n\in \N$, $x\in V_n$ and $r>0$, we define a function $f_{x,r}^{(n)}: V_n \to W_n \subset \Cliff(W_n)$ by:
\[
f_{x,r}^{(n)}(v) = \varphi_r(\|v-x\|)\psi\big(\frac{v-x}{\|v-x\|}\big)
\]
where $ \psi:S(\bigcup_{n=1}^\infty V_n) \rightarrow S(\bigcup_{n=1}^\infty W_n) $ is the uniformly continuous function from the definition of Property (H) for $V$, and $ \varphi_r : [0,\infty) \to [0,\infty) $ is defined by 
\[ 
\varphi_r(t) =
	\begin{cases}
		0 & \text{if } 0\le t \le r/2,\\
		(2t/r)-1 & \text{if } r/2\le t\le r,\\
		1 & \text{if } t \geq r.
	\end{cases} 
\]
Let
\[ W_{x,r}= 
	\begin{pmatrix}
		1 & f_{x,r}^{(n)} \\
		0 & 1 
	\end{pmatrix}
    \begin{pmatrix}
    	1 & 0 \\
    	f_{x,r}^{(n)} & 1 
    \end{pmatrix}
    \begin{pmatrix}
	1 & f_{x,r}^{(n)} \\
	0 & 1 
    \end{pmatrix}
    \begin{pmatrix}
	0 & -1 \\
	1 & 0 
    \end{pmatrix}
\]
and we set:
\[
b_{x,r}^{(i)}=W_{x,r} 
	\begin{pmatrix}
		1 & 0 \\
		0 & 0 
	\end{pmatrix}
    W_{x,r}^{-1}.
\]
Then it was shown in \cite{CWY15} that $\{b_{x,r}^{(n)}\}$ satisfies the conditions in Lemma \ref{lem:uniformly almost flatness of b}. 

%
%
%

We now recall the construction of the Bott map $\beta_r$ for the Roe algebra $C^*(P_d(X))$ from \cite[Section 3.3]{CWY15}. As in \cite[Definition 2.4]{CWY15}, we choose $\H_{P_d(X)}:=\ell^2(Z_d; \HH)$ as the $P_d(X)$-module and define $\C_f[P_d(X)]$ to be the $\ast$-algebra of all bounded functions $T: Z_d \times Z_d \to \K$ such that
\begin{enumerate}
  \item for any bounded subset $B\subseteq X$, the set $\big \{(x,y)\in (B\times B) \cap (Z_d\times Z_d) : T(x,y) \neq 0 \big\}$ is finite;
	\item there exists an $ L>0 $ such that $\#\{x\in Z_d: T(x,y) \neq 0\} < L$ and $\# \{y\in Z_d : T(x,y)\neq 0\} < L $ for each $x,y\in Z_d$;
	\item there exists an $ R\ge 0 $ such that $ T(x,y)=0 $ whenever $ d(x,y)>R $ for $x,y\in Z_d$.
\end{enumerate}
Note that $\C_f[X]$ can be regarded as a subalgebra in $\B(\H_{P_d(X)})$, and its norm closure is $\ast$-isomorphic to $C^*(P_d(X))$.

Since $C^*(P_d(X))$ is stable, it suffices to define $\beta_r([P])$ for any idempotent $P \in C^*(P_d(X))$. For any $0 < \epsilon_1 < 1/100$, take $Q \in \C_f[P_d(X)]$ such that $\|P-Q\| < \epsilon_1/(4\|P\|)$. Then $\|Q-Q^2\|<\epsilon_1$ and there exists $R_{\epsilon_1}>0$ such that the propagation of $Q$ does not exceed $R_{\epsilon_1}$. 

For any $\epsilon_2>0$, it follows from Lemma \ref{lem:uniformly almost flatness of b} that there exists a family of $(R_{\epsilon_1},\epsilon_2;r)$-flat idempotents $\{b^{(n)}_v\}_{n\in \N,v\in V_n}$ in $M_2(\A_n^+)$. Denote:
\[
\mathcal{Q}((M_2(\A_n^+))_n) := \frac{\prod_{n=1}^\infty M_2(\A_n^+)}{\bigoplus_{n=1}^\infty M_2(\A_n^+)} \quad \mbox{and} \quad \mathcal{Q}((M_2(\A_n))_n) := \frac{\prod_{n=1}^\infty M_2(\A_n)}{\bigoplus_{n=1}^\infty M_2(\A_n)}.
\]
We define $\widetilde{Q}, \widetilde{Q}_0: Z_d \times Z_d \longrightarrow \mathcal{Q}((M_2(\A_n^+))_n) \otimes \K$ by the formula
\begin{equation}\label{EQ:defn for Q Roe1}
\widetilde{Q}(x,y):=\big[\big( b^{(1)}_{f(x)}, \cdots, b^{(n)}_{f(x)}, \cdots \big)\big] \otimes Q(x,y),
\end{equation}
\begin{equation}\label{EQ:defn for Q Roe2}
\widetilde{Q}_0(x,y):=\big[\big( b_0, \cdots, b_0, \cdots \big)\big] \otimes Q(x,y),
\end{equation}
respectively, for all $(x,y)\in Z_d \times Z_d$. Note that here $b^{(n)}_{f(x)}$ is well defined for sufficiently large $n$ such that $f(x) \in V_n$. It is clear that
\[
\widetilde{Q}, \widetilde{Q}_0 \in \C[P_d(X), \mathcal{Q}((M_2(\A_n^+))_n) \otimes \K]
\]
and
\[
\widetilde{Q} - \widetilde{Q}_0 \in \C[P_d(X), \mathcal{Q}((M_2(\A_n))_n) \otimes \K].
\]

Furthermore, since $X$ has bounded geometry and $Q$ has finite propagation, it follows from Lemma \ref{lem:uniformly almost flatness of b} that $\|{\widetilde{Q}}^2 - \widetilde{Q}\|<1/5$ and $\|\widetilde{Q}_0^2 - \widetilde{Q}_0\|<1/5$ if we choose $\epsilon_1$ and $\epsilon_2$ sufficiently small. Hence the spectrum of either $\widetilde{Q}$ or $\widetilde{Q}_0$ is contained in disjoint neighbourhoods $S_0$ of $0$ and $S_1$ of $1$. Let $\chi:S_0 \sqcup S_1 \to \C$ be a continuous function such that $\chi(S_0)=0$ and $\chi(S_1)=1$. Define $\Theta:=\chi(\widetilde{Q})$ and $\Theta_0:=\chi(\widetilde{Q}_0)$. Then $\Theta$ and $\Theta_0$ are idempotents in $C^*(P_d(X), \mathcal{Q}((M_2(\A_n^+))_n) \otimes \K)$, and $\Theta - \Theta_0$ belongs to the closed two-sided ideal $C^*(P_d(X), \mathcal{Q}((M_2(\A_n))_n) \otimes \K)$.

Finally we apply the \emph{difference construction} in $K$-theory of Banach algebras introduced by Kasparov-Yu \cite{KY06}. More precisely, for a closed two-sided ideal $J$ of a Banach algebra $B$ and two idempotents $p,q \in B$ with $p-q\in J$, we can define a difference element $D(p,q) \in K_0(J)$ associated to $p,q$. We omit the details here but guide the readers to their original paper. Now we define
\[
\beta_r([P]):=D(\Theta, \Theta_0) \in K_0\big( C^*(P_d(X), \q) \big).
\]
The correspondence $[P] \mapsto \beta_r([P])$ extends to a homomorphism 
\[
\beta_r:K_0\big( C^*(P_d(X)) \big) \longrightarrow K_0\big( C^*(P_d(X), \q) \big),
\]
which is called the \emph{Bott map}. By suspension, we define the Bott map for $K_1$-groups in a similar way:
\[
\beta_r:K_1\big( C^*(P_d(X)) \big) \longrightarrow K_1\big( C^*(P_d(X), \q) \big).
\]

\begin{rem}\label{rem:diff between twisted Roe}
Readers might already notice that the above Bott map $\beta_r$ is slightly different from the one constructed in \cite[Section 3.3]{CWY15}, which is defined to be 
\[
\beta=\iota_\ast \circ \beta_r: K_\ast\big( C^*(P_d(X)) \big) \longrightarrow K_\ast\big( C^*(P_d(X), \Q) \big)
\]
where $\iota$ is the monomorphism from Corollary \ref{cor:relation between twisted Roe}. This is due to the observation that the elements $\widetilde{Q}$ and $\widetilde{Q}_0$ defined in (\ref{EQ:defn for Q Roe1}) and (\ref{EQ:defn for Q Roe2}) indeed belong to the subalgebra $\mathcal{Q}((M_2(\A_n^+))_n) \otimes \K$ instead of merely the original one $\mathcal{Q}((M_2(\A_n^+))_n \otimes \K)$. This is crucial to construct Diagram (\ref{EQ:Bott map diagram}) as revealed in the next subsection.
\end{rem}


\subsection{The strongly quasi-local case}\label{ssec:sq Bott}

Now we move to the strongly quasi-local case and define the Bott map $ \beta_{sq} : K_*(C^*_{sq}(P_d(X)))\to K_\ast(C^*_{q}(P_d(X),\q))$. The construction is similar to the Bott map $\beta_r$ for Roe algebras in Section \ref{ssec:Roe Bott}, while the argument is more involved. Again we only consider the $K_0$-case, since the $K_1$-case can be done by suspension.

As in the Roe case, we choose $\H_X:=\ell^2(Z_d, \HH)$ as the fixed $P_d(X)$-module. Since $C^*_{sq}(\H_X)$ is stable, it suffices to define $ \beta_{sq}([P]) $ for any idempotent $ P\in\cs(\H_X) $. In the rest of this subsection, let us fix an idempotent $ P\in\cs(\H_X) $ and an $\epsilon \in (0,\frac{1}{4})$. It follows from Definition \ref{defn:strong quasi-locality} that there exist $\delta,R$ such that for any strongly measurable map $g:P_d(X) \to \K_1$ with $(\delta,R)$-variation, we have $ \|[P\otimes{\Id}_\HH,\Lambda(g)]\|<\epsilon $. We also fix such parameters $\delta$ and $R$. 

Recall that $f: X \to V$ is the coarse embedding and is extend to a continuous coarse embedding $f: P_d(X) \to V$ by affine extension (see the beginning of Section \ref{sec:twisted alg}). We fix two proper functions $\rho_\pm: [0,\infty) \to [0,\infty)$ such that for any $x,y \in P_d(X)$ we have:
\[
\rho_-(d(x,y)) \leq \|f(x) - f(y)\| \leq \rho_+(d(x,y)).
\]
Moreover, Lemma \ref{lem:uniformly almost flatness of b} provides an $r>0$ and a family of idempotents ${\{b_{x}^{(n)}\}}_{n\in\N,x\in V_n}$ which is $(\rho_+(R),\delta; r)$-flat. Again let us fix such $r$ and ${\{b_{x}^{(n)}\}}$.

Denote $E_2:=\ell^2(Z_d; \mathcal{Q}((M_2(\A_n))_n) \otimes \K)$ the Hilbert module over $\mathcal{Q}((M_2(\A_n))_n) \otimes \K$, and fix a faithful representation $\pi: \mathcal{Q}((M_2(\A_n))_n) \to \B(\H)$. Then $\pi \otimes i_{\K}: \mathcal{Q}((M_2(\A_n))_n) \otimes \K \to \B(\H \otimes \HH)$ is also a faithful representation, where $i_{\K}: \K \to \B(\HH)$ is the inclusion map. For any $ T\in\cs(\H_X) \subset \B(\ell^2(Z_d;\HH))$, we consider its \emph{amplification} $a(T) \in \B(\ell^2(Z_d;\H \otimes \HH))$ defined by $a(T)_{xy}=\Id_{\H} \otimes T_{xy}$ for any $x,y\in Z_d$. It follows from Lemma \ref{lem:the norm on Hilbert module} that there is an operator $\Psi(a(T)) \in \B(E_2)$ defined by (\ref{EQ:norm control}). Note that $\Psi(a(T))_{xy} = \Id_{\mathcal{Q}((M_2(\A_n))_n)} \otimes T_{xy}$ for any $x,y\in Z_d$.

On the other hand, we define a diagonal operator $B \in \B(E_2)$ by
\[
B_{xy}:=
\begin{cases}
~\big[\big( b^{(1)}_{f(x)}, \cdots, b^{(n)}_{f(x)}, \cdots \big)\big] \otimes \Id_{\HH}, & y=x; \\ 
~0, & \mbox{otherwise}
\end{cases}
\]
where $b^{(n)}_{f(x)}$ is well-defined for $n$ large enough such that $f(x) \in V_n$. Note that each $\big[\big( b^{(1)}_{f(x)}, \cdots, b^{(n)}_{f(x)}, \cdots \big)\big] \in \mathcal{M}\big(\mathcal{Q}((M_2(\A_n))_n)\big)$, hence $B$ is well-defined. Similarly, we define $B_0 \in \B(E_2)$ by
\[
(B_0)_{xy}:=
\begin{cases}
[(b_0, \cdots, b_0, \cdots )] \otimes \Id_{\HH}, & y=x; \\ 
~0, & \mbox{otherwise}.
\end{cases}
\]
It is clear that $B$ and $B_0$ are idempotents in $\B(E_2)$, and $\|B\|=\|B_0\|=1$.

Now for the given idempotent $P\in C^*_{sq}(\H_X)$, we define $\widetilde{P}:=B \cdot \Psi(a(P)) \in \B(E_2)$ and $\widetilde{P}_0:=B_0 \cdot \Psi(a(P)) \in \B(E_2)$. It is clear that for any $x,y\in Z_d$, we have:
\begin{equation}\label{EQ:defn for Q sq1}
\widetilde{P}(x,y)=\big[\big( b^{(1)}_{f(x)}, \cdots, b^{(n)}_{f(x)}, \cdots \big)\big] \otimes P(x,y),
\end{equation}
\begin{equation}\label{EQ:defn for Q sq2}
\widetilde{P}_0(x,y)=\big[\big( b_0, \cdots, b_0, \cdots \big)\big] \otimes P(x,y).
\end{equation}
Readers may compare (\ref{EQ:defn for Q sq1}), (\ref{EQ:defn for Q sq2}) with (\ref{EQ:defn for Q Roe1}), (\ref{EQ:defn for Q Roe2}).

We define the $C^*$-algebra $\cq(P_d(X),\2Q)$ similarly to Definition \ref{defn:twisted strongly quasi-local algebra} except that we use $\mathcal{Q}((M_2(\A_n))_n) \otimes \K$ to replace $\q$ therein. It is clear that $\cq(P_d(X),\2Q)$ is a subalgebra of $\B(E_2)$.  
We also define $\cq(P_d(X),\+Q)$ to be the $C^*$-subalgebra in $\B(E_2)$ generated by $\cq(P_d(X),\2Q)$ and $\big\{B_0 \cdot \Psi(a(Q)): Q\mbox{~is~an~idempotent~in~}C^*_{sq}(\H_X)\big\}$. 
It is easy to check that $\cq(P_d(X),\2Q)$ is a two-sided $\ast$-ideal in $\cq(P_d(X),\+Q)$.

The following lemma is a crucial step to define the Bott map:

\begin{lem}\label{lem:Bott sq P-element}
With the same notation as above, we have:
\[
\widetilde{P}, \widetilde{P}_0\in \cq(P_d(X),\+Q) \mbox{~and~} \widetilde{P} - \widetilde{P}_0 \in \cq(P_d(X),\2Q).
\]
\end{lem}


%
%
%
%

\begin{proof}
It suffices to prove that $\widetilde{P} - \widetilde{P}_0 \in \cq(P_d(X),\2Q)$. We need to check condition (1)-(4) in Definition \ref{defn:twisted strongly quasi-local algebra} for $\widetilde{P} - \widetilde{P}_0$. Note that (1) holds trivially.

For (3): Since $P$ is quasi-local, it follows that for any $\epsilon'>0$ there exists $r_1>0$ such that for any Borel subsets $C, D\subseteq P_d(X)$ with $d(C,D)>r_1$, then $\|\chi_C P \chi_D\| < \epsilon'/\|B-B_0\|$. Hence
\begin{align*}
& \|\chi_C (B-B_0) \Psi(a(P)) \chi_D\| = \|(B-B_0) \chi_C \Psi(a(P)) \chi_D\| \\
&\leq  \|B-B_0\|\cdot \|\Psi(a(\chi_C P \chi_D))\| 
 \leq  \|B-B_0\|\cdot \|\chi_C P \chi_D\| < \epsilon',
\end{align*}
where we use Lemma \ref{lem:the norm on Hilbert module} in the second last inequality.

For (4): Given $\epsilon'>0$, we take $r_2=r$.
For any Borel subset $ C\subseteq P_d(X) $ and $h\in C_b(V)_1$ with $ d(f(C),\supp(h))>r_2 $, we have:
\[
\chi_C (\widetilde{P} - \widetilde{P}_0) h = \chi_C (B - B_0)\Psi(a(P)) h = h (B - B_0) \chi_C \Psi(a(P)).
\]
Since $ \supp\big(b_{f(x)}^{(n)}-b_0\big)\subseteq B_{V_n}(f(x),r) $ for any $x \in P_d(X)$ and $n$ large enough such that $f(x) \in V_n$, we obtain that $h (B - B_0) \chi_C =0$ due to the choice of $r_2$. 

It remains to prove (2). Fixing a compact subset $K \subset P_d(X)$, we aim to show that $\chi_K(\widetilde{P} - \widetilde{P}_0)$ and $(\widetilde{P} - \widetilde{P}_0)\chi_K$ belong to $\K(E_2)$. Here we only consider $\chi_K(\widetilde{P} - \widetilde{P}_0)$ since the other is similar. Note that
\[
\chi_K(\widetilde{P} - \widetilde{P}_0) = (B-B_0)\chi_K\Psi(a(P)) = (B- B_0)\cdot \Psi(a(\chi_K P)).
\]
Since $P \in C^*_{sq}(P_d(X))$, then $\chi_K P$ is compact. Without loss of generality, we may assume that $\chi_K P$ is rank-one, i.e., $(\chi_K P)(\zeta)=\xi\langle\eta, \zeta \rangle$ where $\xi, \eta, \zeta \in \ell^2(Z_d; \HH)$. We fix a unit vector $v_0 \in \HH$. Denote the rank one operator $u \mapsto v\langle w, u \rangle$ by $\theta_{v,w}\in \B(\HH)$ where $u,v,w \in \HH$. Direct calculation shows that for any $x,y\in Z_d$, we have:
\[
\big((B-B_0) \cdot \Psi(a(\chi_K P))\big)_{xy} = \big[\big( b^{(1)}_{f(x)} - b_0, \cdots, b^{(n)}_{f(x)} - b_0, \cdots \big)\big]\otimes \theta_{\xi(x),v_0} \theta_{v_0,\eta(y)}.
\]
Let $\{q_i\}_{i\in I}$ be a net of approximating units for the $C^*$-algebra $\mathcal{Q}((M_2(\A_n))_n)$. We consider two vectors $\Xi_1, \Xi_2\in E_2$ defined by
\[
\Xi_1:=\sum_{x\in Z_d} \big(\big[\big( b^{(1)}_{f(x)} - b_0, \cdots, b^{(n)}_{f(x)} - b_0, \cdots \big)\big]\otimes \theta_{\xi(x),v_0}\big) [x],
\]
and 
\[
\Xi_2^{(i)}:=\sum_{x\in Z_d} \big(q_i \otimes \theta_{\eta(x),v_0} \big) [x].
\]
For each $i\in I$, denote the operator $\Theta_i: \Xi \mapsto \Xi_1\langle \Xi_2^{(i)}, \Xi \rangle$ where $\Xi\in E_2$. Then clearly $\Theta \in \K(E_2)$. Moreover, we obtain that the net $\{\Theta_i\}_{i\in I}$ converges to $\chi_K(\widetilde{P} - \widetilde{P}_0)$ in norm, which concludes the proof. 
\end{proof}

Our last ingredient to construct the Bott map is the following:

\begin{lem}\label{lem:approximating idempotent}
With the same notation as above, we have $\|\widetilde{P}^2 - \widetilde{P}\| \leq \epsilon$ and $\|\widetilde{P}_0^2 - \widetilde{P}_0\| \leq \epsilon$.
\end{lem}

\begin{proof}
We only deal with the case of $\widetilde{P}$ while the other is trivial. Recall that $\widetilde{P} = B \cdot \Psi(a(P))$ and $B, P$ are idempotents, hence it suffices to control $\|[B, \Psi(a(P))]\|$.

Recall that $\pi: \mathcal{Q}((M_2(\A_n))_n) \to \B(\H)$ is the faithful representation fixed above. We consider the following map $g: P_d(X) \to \B(\H)_1$ defined by
\[
g(x)= \pi\big(\big[\big( b^{(1)}_{f(x)}, \cdots, b^{(n)}_{f(x)}, \cdots \big)\big] \big), \quad x\in P_d(X).
\]
Since ${\{b_{z}^{(n)}\}}_{n\in\N,z\in V_n}$ is $(\rho_+(R),\delta; r)$-flat for the $R, \delta$ chosen above, we obtain that the map $g$ has $(\delta, R)$-variation on $P_d(X)$. Moreover, it is easy to check that $g$ is continuous. Define a diagonal operator $G\in \B(\ell^2(Z_d;\H \otimes \HH))$ by
\[
G_{xy}:=\begin{cases}
		g(x)\otimes\Id_{\HH}, & \text{if } x=y;\\
		0, & \text{otherwise}.
	\end{cases} 
\]
Then it is clear that $\Psi([G,a(P)]) = [B, \Psi(a(P))]$. Hence from Lemma \ref{lem:the norm on Hilbert module}, it suffices to control $\|[G,a(P)]\|_{\B(\ell^2(Z_d;\H \otimes \HH))}$.

For any finite rank projection $q \in \B(\H)$, denote $\hat{q}:=\Id_{\ell^2(Z_d)} \otimes q \otimes \Id_{\HH} \in \B(\ell^2(Z_d;\H \otimes \HH))$. We regard $q(\H)$ as a subspace in $\HH$. Consider the continuous map $g_q: P_d(X) \to \B(q\H)_1 \subseteq \K(\HH)_1$ defined by $g_q(x):= q \cdot g(x) \cdot q$, which also has $(\delta, R)$-variation. It is straightforward to check that for any $x,y\in Z_d$, we have:
\[
\big(\Lambda(g_q)\big)_{xy} =
\begin{cases}
		\Id_{\HH} \otimes g_q(x), & \text{if } x=y;\\
		0, & \text{otherwise}
	\end{cases} 
\]
where $\Lambda(g_q)$ is defined as in Definition \ref{defn:definition of Lambda general case}. Hence we obtain:
\[
\big\|\hat{q}[G,a(P)] \hat{q}\big\| \leq \big\|[P \otimes \Id_{\HH}, \Lambda(g_q)]\big\| < \epsilon.
\]
Finally note that
\[
\big\|[G,a(P)] \big\| = \sup\big\{\big\|\hat{q}[G,a(P)] \hat{q}\big\|: q \mbox{~is~a~finite~rank~projection~in~}\B(\H)\big\},
\]
hence we conclude the proof.
\end{proof}

Consequently, we obtain from Lemma \ref{lem:approximating idempotent} that the spectrum of either $\widetilde{P}$ or $\widetilde{P}_0$ is contained in disjoint neighbourhoods $S_0$ of $0$ and $S_1$ of $1$. Let $\chi:S_0 \sqcup S_1 \to \C$ be a continuous function such that $\chi(S_0)=0$ and $\chi(S_1)=1$. Define $\Theta:=\chi(\widetilde{P})$ and $\Theta_0:=\chi(\widetilde{P}_0)$. Then it follows from Lemma \ref{lem:Bott sq P-element} that $\Theta$ and $\Theta_0$ are idempotents in $C_q^*(P_d(X), \mathcal{Q}((M_2(\A_n^+))_n) \otimes \K)$, and $\Theta - \Theta_0$ belongs to the closed two-sided ideal $C_q^*(P_d(X), \mathcal{Q}((M_2(\A_n))_n) \otimes \K)$.

Finally we apply the difference construction from \cite{KY06} as in the Roe case, and define 
\[
\beta_{sq}([P]):=D(\Theta, \Theta_0) \in K_0\big( C_q^*(P_d(X), \q) \big).
\]
The correspondence $[P] \mapsto \beta_{sq}([P])$ extends to a homomorphism 
\[
\beta_{sq}: K_0\big( C_{sq}^*(P_d(X)) \big) \longrightarrow K_0\big( C_q^*(P_d(X), \q) \big),
\]
which is also called the \emph{Bott map}. From the constructions above, it is easy to see that Diagram (\ref{EQ:Bott map diagram}) commutes.

\section{local isomorphism}\label{sec:loc isom}

In this section, we relate the $K$-theories of restricted twisted Roe algebras and their quasi-local counterparts. The main result is the following:

\begin{thm}\label{thm:isomorphism between twisted Roe algebra and twisted quasi-local algebra in K-theory} 
Let $X$ be a discrete metric space with bounded geometry which admits a coarse embedding into a Banach space with Property (H). Then for each $d \geq 0$, the natural inclusion $i^{\A}$ from (\ref{EQ: iA}) induces isomorphisms on $K$-theories, \emph{i.e.}, the following
\[
i_*^{\A}:K_*\big(\c(P_d(X),\q)\big) \longrightarrow K_*\big(\cq(P_d(X),\q)\big)
\]
is an isomorphism for $\ast=0,1$.
\end{thm}

The idea of the proof is to decompose the twisted algebras into smaller subalgebras which are easier to handle and apply the Mayer-Vietoris sequence to paste them together. This follows the outline of \cite[Section 4]{CWY15}, which originates from \cite{Yu00}. We use the same notation introduced in Section \ref{sec:twisted alg} and \ref{sec:Bott}, and fix a $d \geq 0$ throughout the section.



The following notion is inspired by \cite[Definition 4.2]{CWY15}:

\begin{defn}\label{defn:the generating set of local twisted Roe algebra}
Let $ O\subseteq V $ be an open subset of $V$. Define $ \C{[P_d(X),\q]}_O $ to be the $ \ast $-subalgebra of $ \C[P_d(X),\q] $ generated by elements $T \in \C[P_d(X),\q]$ such that for any $ x,y\in Z_d $, we have:
\[
\supp(h_n)\subseteq O\cap V_n
\]
where $ T_{xy}=[(h_1,\cdots,h_n,\cdots)]\in \q $. We define $ \c{(P_d(X),\q)}_O $ to be the norm closure of $ \C{[P_d(X),\q]}_O $ in $ \B(E) $.
\end{defn}


Similarly, we introduce the following for quasi-local algebras:

\begin{defn}\label{defn:local twisted strongly quasi-local algebra}
Let $ O\subseteq V $ be an open subset of $V$. Define $ {\cq(P_d(X),\q)}_O $ to be the $ \ast $-subalgebra of $ \cq(P_d(X),\q) $ generated by elements $T \in \cq(P_d(X),\q)$ such that for any $ x,y\in Z_d $, we have:
\[
\supp(h_n)\subseteq O\cap V_n
\]
where $ T_{xy}=[(h_1,\cdots,h_n,\cdots)]\in \q $. For a Borel subset $Y \subseteq P_d(X)$, we also denote by $ {\cq(Y,\q)}_O $ the $C^*$-algebras consisting of elements in $\cq(P_d(X),\q)$ with support in $Y \times Y$.
\end{defn}

Note that $ \c{(P_d(X),\q)}_O $ is a two-sided $\ast$-ideal of $ \c(P_d(X),\q) $. It is also easy to see that $ {\cq{(P_d(X),\q)}}_O $ is norm closed, and is a two-sided $\ast$-ideal of $ \cq(P_d(X),\q) $. Moreover, it is clear that we have the inclusion
\[
i^{\mathcal{A}}: C^*(P_d(X),\q)_O \hookrightarrow C^*_q(P_d(X),\q)_O.
\]

%
%
%
%
%
%
%
%
%
%
%

We need to consider some special open subsets $O$ as follows:

\begin{defn}\label{defn:(gamma,r)-separate}
Let $ \Gamma\subseteq X$ and $r>0 $. An open subset $ O\subseteq V $ is said to be \emph{$(\Gamma,r) $-separate} if the following holds:
\begin{enumerate}
		\item $ O=\bigsqcup_{\gamma\in\Gamma}O_{\gamma} $, where $O_\gamma \subset V$ such that $ O_{\gamma}\cap O_{\gamma'}=\emptyset$ whenever $ \gamma \neq \gamma' $;
		\item for any $\gamma \in \Gamma$, we have $O_{\gamma} \subseteq B_{V}(f(\gamma),r) $.
\end{enumerate}
\end{defn}

The following is the key building block to achieve Theorem \ref{thm:isomorphism between twisted Roe algebra and twisted quasi-local algebra in K-theory}:

\begin{prop}\label{prop:isomorphism between local twisted Roe algebra and local twisted quasi-local algebra in K-theory}
Suppose $\Gamma \subset X$ and $r>0$. Then for any $(\Gamma, r)$-separate open subset $O \subset V$, we have that
\[
i^\A_\ast: K_*({\cq(P_d(X),\q)}_O)\cong K_*(\c{(P_d(X),\q)}_O)
\]
is an isomorphism for $\ast=0,1$.
\end{prop}

The proof is divided into several lemmas. We start with some extra notation. Let $ O=\bigsqcup_{\gamma\in\Gamma}O_{\gamma} $ be a $(\Gamma, r)$-separate open subset of $V$. For each $\gamma\in \Gamma$, we define $ (\q)_{O_\gamma} $ to be a $C^*$-subalgebra of $ \q $ generated by $[(h_1, \cdots, h_n, \cdots)] \in \q$ such that $\supp(h_n)\subseteq O_\gamma \cap V_n$ for all $n\in \N$. 

For each $\gamma \in \Gamma$ and $S\geq0$, we denote $Y_\gamma(S):=\{x\in P_d(X): d(x,\gamma) \leq S\}$.
We also fix two proper functions $\rho_\pm: [0,\infty) \to [0,\infty)$ such that for any $x,y \in P_d(X)$ we have:
\[
\rho_-(d(x,y)) \leq \|f(x) - f(y)\| \leq \rho_+(d(x,y)).
\]



The following lemma is actually from \cite[Lemma 4.5]{CWY15}, hence we omit the proof.

\begin{lem}\label{lem:local isomorphism for local twisted Roe algebra}
Let $ O=\bigsqcup_{\gamma\in\Gamma}O_{\gamma} $ be a $(\Gamma, r)$-separate open subset of $V$. Then
\[ 
\c{(P_d(X),\q)}_O \cong \lim_S \prod_{\gamma\in\Gamma} \c(Y_\gamma(S))\otimes{(\q)}_{O_\gamma}.
\]
\end{lem}

For the quasi-local case, we have an analogous decomposition. Since the proof is similar, we only provide a sketch here.
\begin{lem}\label{lem:local isomorphism for local twisted quasi-local algebra 1}
Let $ O=\bigsqcup_{\gamma\in\Gamma}O_{\gamma} $ be a $(\Gamma, r)$-separate open subset of $V$. Then
\[ 
\cq(P_d(X),\q)_O \cong \lim_S \prod_{\gamma\in\Gamma} \cq\big(Y_\gamma(S),\q\big)_{O_\gamma}.
\]
\end{lem}

\begin{proof}[Sketch of proof]
For any $ T \in \cq(P_d(X),\q)_O $ and $\gamma\in \Gamma$, it is clear that $T\chi_{O_\gamma} \in \cq(P_d(X),\q)_{O_\gamma}$. Given $\epsilon>0$, it follows from Definition \ref{defn:twisted strongly quasi-local algebra} together with Remark \ref{rem:twisted sq alg star condition} that there exist $r_2>0$ such that for any Borel subset $ C\subseteq P_d(X) $ and $h\in C_b(V)_1$ with $ d(f(C),\supp(h))\geq r_2 $, we have $ \|\chi_CTh\|<\epsilon/2$ and $\|h T \chi_C\| < \epsilon/2$.

Take $S:=\sup\{s\in [0,\infty): \rho_-(s) \leq r+r_2\}$. Then for every $\gamma\in \Gamma$, we have
\[
f(X \setminus Y_\gamma(S)) \subseteq V \setminus B(f(\gamma), r+r_2),
\]
which implies that $d(f(X \setminus Y_\gamma(S)), O_\gamma) \geq r_2$. Hence we obtain:
\[
\|T\chi_{O_\gamma} - \chi_{Y_\gamma(S)}T \chi_{O_\gamma}\| = \| \chi_{X \setminus Y_\gamma(S)} T \chi_{O_\gamma} \| < \epsilon/2,
\]
and
\[
\|\chi_{O_\gamma}T - \chi_{O_\gamma} T \chi_{Y_\gamma(S)}\| = \| \chi_{O_\gamma}T \chi_{X \setminus Y_\gamma(S)} \| < \epsilon/2.
\]
Note that $T\chi_{O_\gamma} = \chi_{O_\gamma} T$, so we obtain:
\begin{align*}
\|T\chi_{O_\gamma} - \chi_{Y_\gamma(S)}(T \chi_{O_\gamma})\chi_{Y_\gamma(S)}\| &\leq \|T\chi_{O_\gamma} - \chi_{Y_\gamma(S)}T \chi_{O_\gamma}\| + \|\chi_{Y_\gamma(S)}T \chi_{O_\gamma} - \chi_{Y_\gamma(S)}(T \chi_{O_\gamma})\chi_{Y_\gamma(S)}\|\\
& \leq \epsilon/2 + \|\chi_{O_\gamma}T - \chi_{O_\gamma}T \chi_{Y_\gamma(S)}\| \leq \epsilon/2 +  \epsilon/2 = \epsilon.
\end{align*}
Therefore, we obtain a well-defined homomorphism 
\[
\Phi: \cq(P_d(X),\q)_O \longrightarrow \lim_S \prod_{\gamma\in\Gamma} \cq\big(Y_\gamma(S),\q\big)_{O_\gamma}
\]
defined by 
\[
T \mapsto (T\chi_{O_\gamma})_{\gamma\in \Gamma} = \lim_{S \to \infty} \big(\chi_{Y_\gamma(S)}(T \chi_{O_\gamma})\chi_{Y_\gamma(S)}\big)_{\gamma\in \Gamma}.
\]

On the other hand, note that $T=\sum_{\gamma\in \Gamma} T \chi_{O_\gamma}$ where the summation converges in the strong $*$-operator topology. Moreover, the above argument shows that
\[
T = \sum_{\gamma\in \Gamma} T \chi_{O_\gamma} = \lim_{S \to \infty} \sum_{\gamma\in \Gamma} \chi_{Y_\gamma(S)}(T \chi_{O_\gamma})\chi_{Y_\gamma(S)}.
\]
Therefore, it is not hard to check that $\Phi$ provides a $\ast$-isomorphism. Details are omitted here.
\end{proof}

Furthermore, we have the following:

\begin{lem}\label{lem:local isomorphism for local twisted quasi-local algebra 2}
For each $S \geq 0$ and $\gamma\in \Gamma$, we have 
\[
{\cq(Y_\gamma(S),\q)}_{O_\gamma}\cong\cq(Y_\gamma(S))\otimes{(\q)}_{O_\gamma}.
\]
\end{lem}

\begin{proof}
For any $T \in {\cq(Y_\gamma(S),\q)}_{O_\gamma}$, we have $T = \chi_{Y_\gamma(S)} T \in \K(E)$. Hence it is easy to see that
\begin{align*}
{\cq(Y_\gamma(S),\q)}_{O_\gamma}&=\K(\ell^2(Y_\gamma\cap Z_d,\q)) \\
&\cong \K(\ell^2(Y_\gamma(S)\cap Z_d))\otimes(\q)\\
& \cong \cq(Y_\gamma(S))\otimes{(\q)}_{O_\gamma},
\end{align*}
which concludes the proof.
\end{proof}

%

\begin{proof}[Proof of Proposition \ref{prop:isomorphism between local twisted Roe algebra and local twisted quasi-local algebra in K-theory}]
Note that for each $\gamma\in \Gamma$, we have:
\[
\cq(Y_\gamma(S)) \cong \K(\ell^2(Y_\gamma(S)\cap Z_d)) \cong C^*(Y_\gamma(S)).
\]
Hence the proof follows directly from Lemma \ref{lem:local isomorphism for local twisted Roe algebra}, Lemma \ref{lem:local isomorphism for local twisted quasi-local algebra 1} and Lemma \ref{lem:local isomorphism for local twisted quasi-local algebra 2}.
\end{proof}


%
%

Having established Proposition \ref{prop:isomorphism between local twisted Roe algebra and local twisted quasi-local algebra in K-theory}, we now apply a Mayer-Vietoris sequence argument to obtain Theorem \ref{thm:isomorphism between twisted Roe algebra and twisted quasi-local algebra in K-theory}. To achieve, we need two extra lemmas. The first is actually \cite[Lemma 4.8]{CWY15} together with its quasi-local version. The proof is similar, hence omitted.

\begin{lem}\label{lem:a direc limit about local twisted algebra}
Let $N \in \N$ and $\Gamma_1, \cdots, \Gamma_N$ be $N$ mutually disjoint subsets of $X$. For any $r>0$ and $j \in \{1,2,\cdots, N\}$, let
\[
O_{r,j} :=\bigcup_{\gamma\in\Gamma_j}B_V(f(\gamma),r).
\]
Then for any $r_0>0$ and $k \in \{1,2,\cdots, N-1\}$, we have the following:
	\begin{align*}
		  \mathrm{(1)} \qquad \lim_{r<r_0, r \to r_0}&{\c(P_d(X),\q)}_{O_{r,k}} + \lim_{r<r_0, r \to r_0}{\c(P_d(X),\q)}_{\bigcup_{j=1}^{k-1}O_{r,j}} \\
		&=\lim_{r<r_0, r \to r_0}{\c(P_d(X),\q)}_{\bigcup_{j=1}^{k}O_{r,j}};\\
		 \mathrm{(2)} \qquad  \lim_{r<r_0, r \to r_0}&{\c(P_d(X),\q)}_{O_{r,k}} \cap \lim_{r<r_0, r \to r_0}\c(P_d(X),\q)_{\bigcup_{j=1}^{k-1}O_{r,j}}\\
		 &=\lim_{r<r_0, r \to r_0}{\c(P_d(X),\q)}_{O_{r,k}\cap(\bigcup_{j=1}^{k-1}O_{r,j})};\\
		 \mathrm{(3)} \qquad  \lim_{r<r_0, r \to r_0}&{\cq(P_d(X),\q)}_{O_{r,k}} + \lim_{r<r_0, r \to r_0}{\cq(P_d(X),\q)}_{\bigcup_{j=1}^{k-1}O_{r,j}}\\
		 &= \lim_{r<r_0, r \to r_0}{\cq(P_d(X),\q)}_{\bigcup_{j=1}^{k}O_{r,j}};\\
		 \mathrm{(4)} \qquad  \lim_{r<r_0, r \to r_0}&{\cq(P_d(X),\q)}_{O_{r,k}} \cap \lim_{r<r_0, r \to r_0}{\cq(P_d(X),\q)}_{\bigcup_{j=1}^{k-1}O_{r,j}}\\
		 &=\lim_{r<r_0, r \to r_0}{\cq(P_d(X),\q)}_{O_{r,k}\cap(\bigcup_{j=1}^{k-1}O_{r,j})}.
	\end{align*}
\end{lem}

\begin{lem}\label{lem:the direct limit of local twisted quasi-local algebra and local twisted Roe algebra}
For any $r>0$, denote
\[
O_r:= \bigcup_{x\in X}B_V(f(x),r) = \Nd_r\big(\mathrm{Im} f\big)
\]
where $X \to V$ is the coarse embedding. Then we have:
\begin{align*}
 \lim_{r\to+\infty}{\c(P_d(X),\q)}_{O_r}&\cong\c(P_d(X),\q);\\
 \lim_{r\to+\infty}{\cq(P_d(X),\q)}_{O_r}&\cong\cq(P_d(X),\q).
\end{align*}
\end{lem}

\begin{proof}
The Roe case is already known from \cite{CWY15}, hence we only show the second. Given $T \in \cq(P_d(X),\q)$ and $\epsilon>0$, there exists $r>0$ such that for any Borel subset $ C\subseteq P_d(X) $ and $h\in C_b(V)_1$ with $ d(f(C),\supp(h))>r $, we have $ \|\chi_CTh\|<\epsilon $. In particular, taking a function $h \in C_b(V)_1$ such that $h|_{O_r} = 0$ and $h|_{O_{2r}}=1$, we obtain:
\[
\|T - T(1-h)\| = \|Th\| = \|\chi_X T h\| < \epsilon.
\]
Note that $T(1-h) \in {\c(P_d(X),\q)}_{O_{2r}}$, hence we conclude the proof.
\end{proof}

Now we are in the position to prove Theorem \ref{thm:isomorphism between twisted Roe algebra and twisted quasi-local algebra in K-theory}:

\begin{proof}[Proof of Theorem \ref{thm:isomorphism between twisted Roe algebra and twisted quasi-local algebra in K-theory}]
For any $r>0$, denote $O_r:= \bigcup_{x\in X}B_V(f(x),r)$. From Lemma \ref{lem:the direct limit of local twisted quasi-local algebra and local twisted Roe algebra}, it suffice to show that for each $r_0>0$, the map
\[
i^\A_\ast: K_\ast \big( \lim_{r<r_0, r \to r_0}{\c(P_d(X),\q)}_{O_r} \big) \longrightarrow K_\ast \big( \lim_{r<r_0, r \to r_0}{C^*_q(P_d(X),\q)}_{O_r} \big)
\]
is an isomorphism. 

Since $X$ has bounded geometry, there follows (see, \emph{e.g.}, \cite[Lemma12.2.3]{willett2020higher}) that there exists an $N \in \N$ such that
\begin{enumerate}
 \item $X = \bigsqcup_{j=1}^N \Gamma_j$ for some $\Gamma_j \subset X$ with $\Gamma_j \cap \Gamma_{j'} = \emptyset$ whenever $j \neq j'$;
 \item for each $j\in \{1,2,\cdots, N\}$ and $\gamma \neq \gamma'$ in $\Gamma_j$, we have
 \[
 \|f(\gamma) - f(\gamma')\|_V > 2r_0.
 \]
\end{enumerate}
For any $0<r<r_0$ and $j\in \{1,2,\cdots, N\}$, let 
\[
O_{r,j}:=\bigcup_{x\in \Gamma_j} B_V(f(x),r) = \Nd_r(f(\Gamma_j)).
\]
Then $O_r = \bigcup_{j=1}^N O_{r,j}$, and each $O_{r,j}$ or $O_{r,j} \cap \big( \bigcup_{i=1}^{j-1} O_{r,i} \big)$ are $(\Gamma_j,r)$-separate for any $j\in \{1,2,\cdots,N\}$. Applying the Mayer-Vietoris sequence argument finitely many times based on Lemma \ref{lem:a direc limit about local twisted algebra} and using Proposition \ref{prop:isomorphism between local twisted Roe algebra and local twisted quasi-local algebra in K-theory}, we conclude the proof.
\end{proof}

\section{Proof of the main result}\label{sec:proof}

\begin{proof}[Proof of Theorem \ref{thm:main result}]
	
Combining Diagram (\ref{EQ:Bott map diagram}) with the coarse assembly maps, we obtain the following commutative diagram:
\begin{equation*}
	\xymatrix{
		\lim\limits_{d\to\infty}K_*(P_d(X)) \ar[r]^-{\mu}  \ar[rd]^-{\mu_{sq}} &   \lim\limits_{d\to\infty}K_\ast(C^*(P_d(X)))  \ar[r]^-{\beta_r} \ar[d]_{i_\ast} & \lim\limits_{d\to\infty}K_\ast(C^*(P_d(X),\q)) \ar[d]^{i^{\mathcal{A}}_\ast} \\
		& \lim\limits_{d\to\infty}K_\ast(C^*_{sq}(P_d(X)))  \ar[r]^-{\beta_{sq}} & \lim\limits_{d\to\infty}K_\ast(C^*_{q}(P_d(X),\q)),}
\end{equation*}	
Here we use the same notation to denote the homomorphisms after taking $d \to \infty$.
%
%
Recall from \cite[Theorem 1.1]{CWY15} together with Remark \ref{rem:diff between twisted Roe} that the composition $\iota_\ast \circ \beta_r \circ \mu$ is injective, where 
\[
\iota_\ast: K_\ast\big(C^*(P_d(X),\q)\big) \to K_\ast\big(C^*(P_d(X),\Q)\big)
\]
is induced by the inclusion from Corollary \ref{cor:relation between twisted Roe}. Hence we obtain that the composition $\beta_r \circ \mu$ is injective as well. Also note from Theorem \ref{thm:isomorphism between twisted Roe algebra and twisted quasi-local algebra in K-theory} that $i^{\mathcal{A}}_\ast$ is an isomorphism, hence we conclude that $\mu_{sq}$ is injective via a diagram chasing argument.
\end{proof}

\section{The case of non-positively curved manifolds}\label{sec:mfd}

In this final section, we outline the case of metric spaces which can be coarsely embedded into a simply-connected complete Riemannian manifold with non-positively sectional curvature. Recall that the coarse Novikov conjecture was verified for such spaces in \cite{SW07}. Combining the approach from \cite{SW07} with the techniques developed in this paper, we are able to prove the following:



\begin{thm}\label{thm:main result for manifold case}
	Let $ X $ be a bounded geometry metric space which can be coarsely embedded into a simply-connected complete Riemannian manifold of non-positively sectional curvature. Then the strongly quasi-local coarse Novikov conjecture holds for $ X $.
\end{thm}

Since the proof is quite similar to the one for Theorem \ref{thm:main result}, we only provide an outline as follows.


First let us fix some notation. Let $ M $ be a simply-connected complete Riemannian manifold of non-positively sectional curvature. Without loss of generality, we may assume the dimension of $ M $ is $2n $ for some $ n\in\N $. Following the notation from \cite{SW07}, denote the Clifford bundle over $M$ by $ \Cliff(TM) $ and let $\A:=C_0(M,\Cliff(TM))$ be the $C^*$-algebra of continuous functions $a: M \to \Cliff(TM)$ such that $a(x) \in \Cliff(T_x M)$ and vanish at infinity. Let $C_b(M,\Cliff(TM))$ be the $C^*$-algebra of bounded continuous functions $a: M \to \Cliff(TM)$ such that $a(x) \in \Cliff(T_x M)$.

%
%

%
%
%
%

For each $d \geq 0$, we take a countable dense subset $Z_d$ in the Rips complex $P_d(X)$ such that $Z_d \subseteq Z_{d'}$ whenever $d<d'$. Consider the Hilbert module $E:=\ell^2(Z_d;\A\otimes\K)$ over $\A \otimes \K$, where $\K$ is the $C^*$-algebra of compact operators on the separable infinite-dimensional Hilbert space $\HH$. Recall that the associated twisted Roe algebra $C^*(P_d(X), \A \otimes \K)$ was defined in \cite[Definition 3.1]{SW07} as a $C^*$-subalgebra in $\B(E)$. 

Following the idea in Definition \ref{defn:twisted strongly quasi-local algebra}, we introduce the following:

%

\begin{defn}\label{defn:twisted strongly quasi-local algebra for manifold}
	The twisted quasi-local algebra $C^*_q(P_d(X),\A\otimes\K)$ is defined to be the $C^*$-subalgebra in $\B(E)$ consisting of elements $T$ satisfying the following conditions: 
	\begin{enumerate}
		\item each matrix entry of $T$ belongs to $\A\otimes\K$;
		\item for any compact $K \subseteq P_d(X)$, we have that $\chi_K T$ and $T \chi_K$ belongs to $\K(E)$;
		\item for any $\epsilon>0$ there is an $r_1$ such that for any Borel subsets $ C,D\subseteq P_d(X) $ with $ d(C,D)>r_1 $, we have $ \|\chi_CT\chi_D\|< \epsilon $ and $ \|\chi_DT\chi_C\|< \epsilon $;
		\item for any $ \epsilon>0 $, there is an $ r_2 $ such that for any Borel subset $ C\subseteq P_d(X) $ and $h\in C_b(M)_1$ with $ d(f(C),\supp(h))>r_2 $, we have $ \|\chi_CTh\|<\epsilon $.
	\end{enumerate}
\end{defn}

The associated Bott maps $ \beta,\beta_{sq} $ can be defined similarly to the case of Banach spaces with Property (H) in Section \ref{sec:Bott}, except that we replace Lemma \ref{lem:uniformly almost flatness of b} by the following:

\begin{lem}[{\cite[Lemma 5.2]{SW07}}]\label{lem:uniformly almost flatness of b for manifold}
	For any $R>0$ and $\epsilon>0$, there exist $r>0$ and a family of idempotents $ \{b_{x,r}\}_{x\in M} $ in $ M_2(\A^+) $ which is \emph{$ (R,\epsilon;r) $-flat} in the following sense: $ \supp(b_{x,r}-b_0)\subseteq B_M(x,r) $ and for any $x,y\in M$ with $\|x-y\|<R$, we have 
	\[
	\sup_{z\in M}{\big\|b_{x,r}(z) - b_{y,r}(z)\big\|}_{\Cliff(T_zM)\otimes M_2(\C)}<\epsilon,
	\]
where $ \A^+ $ represents the unitization of $ \A $.
\end{lem}

We omit the details and simply conclude that we have the following homomorphisms:
\[\beta: K_\ast\big( C^*(P_d(X)) \big) \longrightarrow K_\ast\big( C^*(P_d(X), \A\otimes\K) \big)\]
\[\beta_{sq}: K_*\big( C_{sq}^*(P_d(X)) \big) \longrightarrow K_*\big( C_q^*(P_d(X), \A\otimes\K) \big).\]

The last piece is an isomorphism result analogous to Theorem \ref{thm:isomorphism between twisted Roe algebra and twisted quasi-local algebra in K-theory}, which states that the natural inclusion 
\[
i^{\A}: \c(P_d(X),\A\otimes\K) \longrightarrow \cq(P_d(X),\A\otimes\K)
\]
induces an isomorphism on $K$-theories. Again the proof is similar hence omitted.

%

Finally, we consider the following commutative diagram:
\begin{equation*}
	\xymatrix{
		\lim\limits_{d\to\infty}K_*(P_d(X)) \ar[r]^-{\mu}  \ar[rd]^-{\mu_{sq}} &   \lim\limits_{d\to\infty}K_\ast(C^*(P_d(X)))  \ar[r]^-{\beta} \ar[d]_{i_\ast} & \lim\limits_{d\to\infty}K_\ast(C^*(P_d(X),\A\otimes\K)) \ar[d]^{i^{\mathcal{A}}_\ast} \\
		& \lim\limits_{d\to\infty}K_\ast(C^*_{sq}(P_d(X)))  \ar[r]^-{\beta_{sq}} & \lim\limits_{d\to\infty}K_\ast(C^*_{q}(P_d(X),\A\otimes\K)),}
\end{equation*}
where $ \beta \circ \mu $ is an isomorphism due to the proof for \cite[Theorem 1.1]{SW07} and $ i_*^{\A} $ is an isomorphism from above. Hence we obtain that $ \mu_{sq} $ is injective, which concludes Theorem \ref{thm:main result for manifold case}.

\bibliographystyle{plain}
\bibliography{bibfileQLCE}
\end{document}